\documentclass[12pt]{amsart}
\usepackage{amsmath,amsthm,amsfonts,amssymb,graphicx,psfrag,dsfont}

\usepackage{hyperref}
\usepackage[alphabetic,initials]{amsrefs}
\usepackage{color}

\newcommand{\red}[1]{{#1}}

\psfrag{intxM}{$[0,\infty) \times M$}
\psfrag{Ww}{$(W,\omega)$}

\theoremstyle{plain}

\newtheorem{thm}{Theorem}
\newtheorem{prop}{Proposition}[section]

\newtheorem{cor}{Corollary}
\newtheorem{lemma}[prop]{Lemma}

\theoremstyle{definition}
\newtheorem{example}[prop]{Example}

\newtheorem{defn}[prop]{Definition}

\theoremstyle{remark}
\newtheorem{remark}[prop]{Remark}

\newcommand{\cov}{\operatorname{cov}}

\newcommand{\Diff}{\operatorname{Diff}}

\newcommand{\ev}{\operatorname{ev}}

\newcommand{\Hom}{\operatorname{Hom}}

\newcommand{\Id}{\operatorname{Id}}

\newcommand{\ind}{\operatorname{ind}}

\newcommand{\interior}{\operatorname{int}}

\newcommand{\muCZ}{\mu_{\operatorname{CZ}}}

\newcommand{\sing}{\operatorname{sing}}

\newcommand{\Symp}{\operatorname{Symp}}

\newcommand{\Vectors}{\operatorname{Vec}}

\newcommand{\wind}{\operatorname{wind}}

\newcommand{\Spp}{\operatorname{Sp}}

\newcommand{\CC}{{\mathbb C}}
\newcommand{\DD}{{\mathbb D}}

\newcommand{\NN}{{\mathbb N}}

\newcommand{\RR}{{\mathbb R}}

\newcommand{\ZZ}{{\mathbb Z}}

\newcommand{\fF}{{\mathcal F}}

\newcommand{\jJ}{{\mathcal J}}

\newcommand{\mM}{{\mathcal M}}
\newcommand{\nN}{{\mathcal N}}

\newcommand{\pP}{{\mathcal P}}
\newcommand{\sS}{{\mathcal S}}

\newcommand{\p}{\partial}

\numberwithin{equation}{section}

\title[Fillable Contact Manifolds and Holomorphic Foliations]{Strongly
Fillable Contact Manifolds and $J$--holomorphic Foliations}
\author{Chris Wendl}
\address{ETH Z\"urich \\ Departement Mathematik, HG G38.1 \\ 
R\"amistrasse 101 \\
8092 Z\"urich \\ 
Switzerland}
\email{wendl@math.ethz.ch}
\urladdr{http://www.math.ethz.ch/~wendl/}
\thanks{Research partially supported by an NSF Postdoctoral Fellowship
(DMS-0603500).}

\begin{document}

\begin{abstract}
We prove that every strong symplectic filling of a planar contact manifold
admits a symplectic Lefschetz fibration over the disk, and
every strong filling of~$T^3$ similarly admits a
Lefschetz fibration over the annulus.  It follows that strongly fillable
planar contact structures are also Stein fillable, and
all \red{strong} fillings of $T^3$ are \red{equivalent up to symplectic
deformation and blowup}.
These constructions result from a compactness theorem for punctured 
$J$--holomorphic curves that foliate a convex symplectic manifold.
We use it also to show that the compactly supported
symplectomorphism group on $T^*T^2$ is contractible, and
to define an obstruction to strong fillability that yields a
non-gauge-theoretic proof of Gay's recent nonfillability result
\cite{Gay:GirouxTorsion} for
contact manifolds with positive Giroux torsion.
\end{abstract}

\maketitle

\tableofcontents

\section{Introduction}
\label{sec:intro}

Let $M$ be a closed, connected and oriented $3$--manifold.  
A (positive, cooriented) 
\emph{contact structure} on $M$ is a $2$--plane distribution of the form 
$\xi = \ker\lambda$, where the \emph{contact form} $\lambda \in \Omega^1(M)$ 
satisfies $\lambda \wedge d\lambda > 0$.  It is a natural question in 
contact geometry to ask whether a given contact manifold $(M,\xi)$ is
symplectically fillable, meaning the following: we say that a compact and
connected symplectic
manifold $(W,\omega)$ with boundary $\p W = M$ is a \emph{weak filling}
of $(M,\xi)$ if $\omega|_\xi > 0$, and it is a \emph{strong filling} if
$\xi = \ker \iota_Y \omega$ for some
vector field $Y$ defined near $\p W$ which points transversely
outward at the boundary and satisfies $L_Y\omega = \omega$.  
\red{If $Y$ extends globally over $W$, then $\iota_Y\omega$ defines a global
primitive of $\omega$ and thus makes $(W,\omega)$ an \emph{exact filling}.}
A still stronger
notion is a \emph{Stein filling} $(W,\omega)$, which comes with an
integrable complex structure $J$ and admits a proper
plurisubharmonic function $\varphi : W \to [0,\infty)$ for which
$\p W$ is a level set, $Y$ is the gradient and $\omega = -d d^{\CC}\varphi$.
We refer to \cites{Etnyre:convexity,OzbagciStipsicz} 
for more details on these notions.

The vector field $Y$ near the boundary of a strong filling is called a
\emph{Liouville vector field}, and it induces a contact form
\red{$\lambda := \iota_Y\omega|_M$}.  As we'll review \red{shortly},
the existence of $Y$ is then equivalent to
the condition that one can smoothly glue the 
\emph{positive symplectization}
$([0,\infty) \times M, d(e^a\lambda))$ to $(W,\omega)$ along $\p W =
\{0\} \times M$; in the language of symplectic field theory
(cf.~\cite{SFTcompactness}), this produces a symplectic cobordism with a
positive cylindrical end.  One can also replace $\lambda$ by a positive
multiple of any
other contact form defining $\xi$ after attaching to $(W,\omega)$ 
a \red{trivial symplectic cobordism} (see \eqref{eqn:cylStein}
below).  In either case, the enlarged \red{symplectic} manifold is \red{exact}
if $(W,\omega)$ is an \red{exact} filling.

In this paper we examine some of the consequences for strong symplectic
fillings and Stein fillings when a subset 
of the contact manifold (or rather its
symplectization) admits foliations by $J$--holomorphic curves.
It turns out that whenever a foliation with certain properties exists,
it can be extended from $[0,\infty) \times M$ to fill the entirety of 
$W$ with embedded $J$--holomorphic curves, forming a symplectic 
Lefschetz fibration (Theorems~\ref{thm:Lefschetz} and~\ref{thm:T3}),
and this decomposition is stable under deformations of the symplectic
structure (Theorem~\ref{thm:stability}).  The existence of such a fibration has
consequences for the topology of the filling, e.g.~for planar contact
structures, it implies that the notions ``strongly fillable''
and ``Stein fillable'' are equivalent (Corollary~\ref{cor:Stein}).  For 
the $3$--torus, our arguments \red{establish a conjecture of Stipsicz
\cite{Stipsicz:gaugeStein} by showing that all minimal strong fillings are
symplectically deformation equivalent, and exact fillings in particular}
are symplectomorphic to star shaped domains in $T^*T^2$
(Theorem~\ref{thm:SteinT3}); moreover, the group of compactly supported
symplectomorphisms on $T^*T^2$ is contractible
(Theorem~\ref{thm:sympT3}).  In other situations, one finds that
the foliation on $W$ produces an obvious contradiction, thus implying that
the contact manifold cannot be strongly fillable 
(Theorem~\ref{thm:obstruction})---this is the case in
particular for any contact manifold with positive Giroux torsion
(Example~\ref{ex:torsion}).
 
\subsubsection*{Acknowledgments}

This work emerged originally out of discussions with Klaus
Niederkr\"uger and subsequently received much valuable
encouragement from John 
Etnyre.  It was the latter in particular who pointed out to me the questions
regarding Giroux torsion and Stein fillability; I'm also grateful to
both John and Paolo Ghiggini for bringing Stipsicz' paper
\cite{Stipsicz:gaugeStein} to my attention after the first version of this 
paper was circulated.  Thanks also to Dietmar Salamon, Ko Honda,
\red{Mark McLean} and especially
Richard Hind for helpful conversations.

\section{Main results}
\label{sec:results}

\subsection{Existence of Lefschetz fibrations and Stein structures}

Recall that a
contact manifold $(M,\xi)$ is called \emph{planar} if it admits an open book
decomposition that supports $\xi$ and has pages of genus zero.  
We refer to \cite{Etnyre:lectures} or \cite{OzbagciStipsicz} 
for the precise definitions;
for our purposes in the statement of the theorem below,
an open book decomposition is a fibration
$\pi : M \setminus B \to S^1$ where the \emph{binding} $B$ is a link in $M$.  
Then the \emph{pages} are the
preimages $\pi^{-1}(t)$ and the condition ``supports $\xi$'' means essentially
that $\xi = \ker\lambda$ for some contact form (a so-called
\emph{Giroux form}) such that $d\lambda$ is
symplectic on the pages and $\lambda$ is positive on the binding.  One can
always ``fatten'' an open book decomposition by expanding $B$ to a tubular
neighborhood $\nN(B)$ and slightly shrinking the pages,
thus deforming $\pi$ to a nearby map
$$
\hat{\pi} : M \setminus \nN(B) \to S^1.
$$
We will use this notation consistently in the following.

Suppose $W$ and $\Sigma$ are compact oriented manifolds of real dimension~$4$
and~$2$ respectively, possibly with boundary.  A \emph{Lefschetz fibration}
$\Pi : W \to \Sigma$ is then a smooth surjective map which is a locally trivial
fibration outside of finitely many critical values $q \in \interior{\Sigma}$,
where each \emph{singular fiber} $\Pi^{-1}(q)$ has a unique critical point,
at which $\Pi$ can be modeled in some choice of complex coordinates by
$\Pi(z_1,z_2) = z_1^2 + z_2^2$.  For $(W,\omega)$ a symplectic manifold, we 
call the Lefschetz fibration \emph{symplectic} if the fibers are symplectic 
submanifolds.  If $q' \in \Sigma$ is close 
to a critical value $q$, then there is a special
circle $C \subset \Pi^{-1}(q')$, called a \emph{vanishing cycle}, such that
the singular fiber $\Pi^{-1}(q)$ can be identified with $\Pi^{-1}(q')$ after
collapsing $C$ to a point.  
(Again, see \cite{OzbagciStipsicz} for precise definitions.)
One says that the Lefschetz fibration is 
\emph{allowable} if all vanishing cycles are homologically nontrivial
in their fibers.

Denote by $\DD \subset \CC$ the closed unit disk, whose boundary
$\p \DD$ is naturally identified with $S^1 = \RR / \ZZ$.  \red{For any symplectic 
manifold $(W,\omega)$ with contact boundary $(M,\xi)$, the
restriction of a symplectic Lefschetz fibration $\Pi : W \to \DD$
over $\p\DD$ defines an open book decomposition supporting~$\xi$
(see \cite{OzbagciStipsicz}*{\S 10.2})}.
One can see in particular that for any Liouville vector field $Y$ near
$\p W$, the induced contact form $\lambda := \iota_Y\omega$ satisfies
$d\lambda > 0$ on each fiber over $\p \DD$.  One can now ask whether the
converse holds: given an open book $\hat{\pi} : M \setminus \nN(B) \to S^1$ 
supporting $\xi$ and a strong filling $W$, does $W$ admit a Lefschetz fibration 
over $\DD$ that restricts to $\hat{\pi}$ on $\p W \setminus \nN(B)$?  
This would be too 
ambitious as stated, as one cannot expect that the contact form induced on $\p W$
will define positive area on the pages of $\hat{\pi}$: this cannot be true in
particular if $\ker \omega|_{\p W}$ is ever tangent to a page.

This problem can be avoided by enlarging the filling so as to induce
different contact forms (but the same contact structure) on the boundary:
if $\iota_Y\omega|_{\p W} = e^f\lambda$ for some contact form $\lambda$ and
smooth function $f : M \to \RR$, then for any other function
$g : M \to \RR$ with $g > f$ one can define the domain
\begin{equation}
\label{eqn:cylStein}
\sS_f^g = \{ (a,m) \in \RR \times M\ |\ f(m) \le a \le g(m)\ \}.
\end{equation}
This yields a symplectic cobordism $(\sS_f^g, d(e^a\lambda))$ with Liouville
vector field $\p_a$, inducing the contact forms 
$\iota_{\p_a} d(e^a\lambda) = e^f\lambda$ and $e^g\lambda$ on its negative
and positive boundaries respectively.  We shall refer to such domains as
\red{\emph{trivial symplectic cobordisms}, and will sometimes also consider
noncompact versions for which $f=-\infty$ or $g = +\infty$.  The following
is proved by a routine computation.

\begin{lemma}
\label{lemma:attachStein}
Assume $(W,\omega)$ is a strong filling of $(M,\xi)$ with Liouville
vector field $Y$ near $\p W$, and $\iota_Y\omega = \lambda'$.  Suppose further
that 
$\lambda$ is a contact form on $M$ and $f : M \to \RR$ is a smooth function
such that $\lambda'|_M = e^f \lambda$.  Then if $\varphi_Y^t$ denotes the flow
of $Y$ for time~$t$, for sufficiently small $\epsilon > 0$,
there is a symplectic embedding
$$
\psi : \left(\sS_{f-\epsilon}^f,d(e^a\lambda)\right) 
\hookrightarrow (W,\omega) : (a,m) \mapsto \varphi^{a - f(m)}_Y(m)
$$
that maps $\p\sS_{-\infty}^f$ to $\p W$ and is a diffeomorphism onto a
closed neighborhood of $\p W$ in $W$.  Moreover $\psi^*\lambda' = e^a\lambda$
and $\psi_*\p_a = Y$.
\end{lemma}

In light of this, one can
smoothly glue any trivial symplectic cobordism of the form
$(\sS_f^g,d(e^a\lambda))$ to $(W,\omega)$, and the enlarged filling
is exact if $(W,\omega)$ is an exact filling.
An important simple example
is the case where $f \equiv 0$ and $g = \infty$: then we are simply
attaching the positive symplectization $([0,\infty) \times M,d(e^a\lambda))$
where $\lambda = \iota_Y\omega|_{\p W}$.  It will often be convenient however
to take nonconstant $f$, so that the contact form appearing in
$d(e^a\lambda)$ may be chosen at will.}

Recall that an \emph{exceptional sphere} in a symplectic $4$--manifold 
$(W,\omega)$ is a symplectically embedded $2$--sphere with self-intersection 
number~$-1$, and $(W,\omega)$ is called \emph{minimal} if it contains no
exceptional spheres.  \red{We can now state the first main result.}

\begin{thm}
\label{thm:Lefschetz}
Suppose $(W,\omega)$ is a strong symplectic filling of a planar contact
manifold $(M,\xi)$, and $\pi : M \setminus B \to S^1$ is a planar open
book supporting~$\xi$.  Then there is an enlarged filling $(W',\omega)$
obtained by attaching a \red{trivial symplectic}
cobordism to $W$, such that $W'$ admits a symplectic Lefschetz fibration
$\Pi : W' \to \DD$ for which $\Pi|_{\p W' \setminus \nN(B)} = \hat{\pi}$.  
Moreover, $\Pi : W' \to \DD$ \red{is} allowable if $W$ 
is minimal.
\end{thm}

The following corollary was pointed out to me by John Etnyre:

\begin{cor}
\label{cor:Stein}
Every strongly fillable planar contact manifold is also Stein fillable.
\end{cor}
\begin{proof}
Suppose $(W,\omega)$ is a strong filling of $(M,\xi)$ and the latter is
planar.  By blowing down as in \cite{McDuff:rationalRuled} and then attaching
a \red{trivial symplectic} cobordism, we can modify $W$ to a minimal filling 
$(\widehat{W},\hat{\omega})$
that admits an allowable symplectic Lefschetz fibration due to
Theorem~\ref{thm:Lefschetz}.  \red{It then follows from Eliashberg's
topological characterization of Stein manifolds \cite{Eliashberg:Stein}
(see also \cites{GompfStipsicz,AkbulutOzbagci:Lefschetz})
that $(\widehat{W},\hat{\omega})$ is symplectically deformation
equivalent to a Stein domain}.
\end{proof}

\red{Recall that by a result of Giroux \cite{Giroux:openbook}, a contact
$3$--manifold is Stein fillable if and only if it admits a supporting open
book whose monodromy is a product of positive Dehn twists.  One can
understand this in the context of Lefschetz fibrations as follows: if
$(W,\omega)$ is a Stein filling of $(M,\xi)$, then it admits a Lefschetz 
fibration over the disk by a result of Loi-Piergallini \cite{LoiPiergallini} 
or Akbulut-Ozbagci \cite{AkbulutOzbagci:Lefschetz}.  The monodromy of the 
resulting open book decomposition of $M$ can then be obtained by composing 
positive Dehn twists along the vanishing
cycles of each singular fiber (see for example \cite{OzbagciStipsicz}).
Conversely, any open book with this property can be realized as the boundary
of some Lefschetz fibration, which admits a Stein structure due to
Eliashberg \cite{Eliashberg:Stein}.  Giroux
asked whether it might in fact be true that \emph{every} open book
of $(M,\xi)$ must have this property when $(M,\xi)$ is Stein fillable.
Theorem~\ref{thm:Lefschetz} implies an affirmative answer at least for
the planar open books:

\begin{cor}
\label{cor:monodromy}
If $(M,\xi)$ is a planar contact manifold, then it is strongly (and thus
Stein) fillable if and only if every supporting planar open book has
monodromy isotopic to a product of positive Dehn twists.
\end{cor}
}

As an immediate consequence \red{of Corollary~\ref{cor:Stein}},
we also obtain a new obstruction to the existence of planar open books:

\begin{cor}
\label{cor:planar}
If $(M,\xi)$ is a contact manifold which is strongly fillable but not
Stein fillable, then it is not planar.
\end{cor}

\begin{remark}
\label{remark:Paolo}
It was not known until recently whether strong and Stein fillability are
equivalent notions: a negative answer was provided by a construction
due to P.~Ghiggini \cite{Ghiggini:strongNotStein} of strongly fillable
contact manifolds that are not Stein fillable.  It follows then from
the above results that Ghiggini's contact structures are not planar.
\end{remark}

The reason here for the restriction to planar contact structures is
that a planar open book can always be
presented as the projection of a $2$--dimensional $\RR$--invariant family of
$J$--holomorphic curves in the symplectization $\RR\times M$.  
This is a special case of
a construction due to C.~Abbas \cite{Abbas:openbook} that relates open 
book decompositions on general contact manifolds to solutions of a
nonlinear elliptic problem, which specifically in the planar case gives
$J$--holomorphic curves.  (An existence proof for the planar
case is \red{also} given in \cite{Wendl:openbook}.)  For analytical reasons,
$J$--holomorphic curves with the desired properties and higher genus
generically cannot exist.\footnote{Hofer pointed out this trouble in 
\cite{Hofer:real} and suggested the aforementioned elliptic problem as a 
potential remedy, but its compactness properties are not yet fully
understood.}
Nonetheless, one can sometimes derive interesting results for non-planar
contact manifolds using other kinds of decompositions with genus zero fibers,
of which the following is an example.

Let $T^3 = S^1 \times S^1 \times S^1 = T^2 \times S^1$ with coordinates 
$(q_1,q_2,\theta)$,
and write the standard contact structure on $T^3$ as $\xi_0 = \ker\lambda_0$
where
$$
\lambda_0 = \cos(2\pi\theta)\ d q_1 + \sin(2\pi\theta)\ d q_2.
$$
This can be identified with the canonical contact form on the unit
cotangent bundle $S^*T^2 \subset T^*T^2$ as follows:
writing points in $T^2$ as $(q_1,q_2)$, we use the natural identification of
$T^*T^2$ with $T^2 \times \RR^2 \ni (q_1,q_2,p_1,p_2)$ and write the
canonical $1$--form as $p_1\ dq_1 + p_2\ dq_2$.  The $3$--torus
is then $S^*T^2 = T^2 \times \p \DD$, with the 
$\theta$--coordinate corresponding to the point 
\red{$(p_1,p_2) = (\cos(2\pi\theta), \sin(2\pi\theta)) \in \p\DD$}, 
and $\lambda_0$ is the
restriction of $p_1\ dq_1 + p_2\ dq_2$ to this submanifold.  The canonical
symplectic form $\omega_0 := dp_1 \wedge dq_1 + dp_2 \wedge dq_2$ on
$T^*T^2 = T^2 \times \RR^2$ can then be written as $- d d^\CC f$ for the proper
plurisubharmonic function $f(q,p) = \frac{1}{2}|p|^2$, thus
$T^2 \times \DD$ is a Stein domain; we shall refer to it as the
\emph{standard} Stein filling of $(T^3,\xi_0)$.  More generally, 
one has the following construction:

\begin{defn}
A \emph{star shaped domain} $\sS \subset T^*T^2$ is a subset of the form
$\{ (q, t f(q,p) \cdot p) \in T^*T^2\ |\ \text{$t \in [0,1]$, 
$(q,p) \in S^*T^2$} \}$ for some smooth function $f : S^*T^2 \to (0,\infty)$.
\end{defn}
Observe that the boundary $\p\sS$ of a star shaped domain is always
transverse to the radial Liouville vector field $p_1 \p_{p_1} + p_2 \p_{p_2}$,
thus $(\sS,\omega_0)$ is clearly an \red{exact} filling of~$T^3$.

Eliashberg showed in \cite{Eliashberg:fillableTorus} that $\xi_0$ is
the \emph{only} strongly fillable contact structure on~$T^3$.
\red{It is not planar due to \cite{Etnyre:planar}*{Theorem~4.1}, as the
standard filling has $b_2^0(T^2 \times \DD) \ne 0$, though
Van Horn-Morris \cite{Vanhorn:thesis} has shown that it does
admit a genus~$1$ open book.}
It also admits the following decomposition, which one 
might think of as a generalization of an open book with planar pages.  
Let $Z = \{ \theta \in \{0,1/2\} \}
\subset T^3$, a union of two disjoint pre-Lagrangian $2$--tori, and define
\begin{equation}
\label{eqn:piT3}
\begin{split}
\pi : T^3 \setminus Z &\to \{0,1\} \times S^1 \\
(q_1,q_2,\theta) &\mapsto 
\begin{cases}
(0,q_2) & \text{ if $\theta \in (0,1/2)$,} \\
(1,q_2) & \text{ if $\theta \in (1/2,1)$.}
\end{cases}
\end{split}
\end{equation}
This is a smooth fibration, and we can think of it
intuitively as a union of
two open book decompositions with cylindrical pages, and the subset
$Z$ playing the role of the binding.  It supports the contact structure in the
sense that $d\lambda_0$ is positive on each fiber, and the fibers have
natural compactifications with boundary in $Z$ such that $\lambda_0$ is
positive on these boundaries.  As with an open book, one can ``fatten''
$Z$ to a neighborhood $\nN(Z)$ and deform $\pi$ to a nearby \red{fibration}
$$
\hat{\pi} : T^3 \setminus \nN(Z) \to \{0,1\} \times S^1,
$$
\red{whose fibers are compact annuli}.

\begin{thm}
\label{thm:T3}
Suppose $(W,\omega)$ is any strong symplectic filling of $(T^3,\xi_0)$.
Then one can attach to $W$ a \red{trivial symplectic} cobordism, producing an
enlarged filling $W'$ that admits a symplectic Lefschetz fibration
$\Pi : W' \to [0,1] \times S^1$ for which
$\Pi|_{\p W' \setminus \nN(Z)} = \hat{\pi}$.  \red{Moreover, every singular
fiber is the union of an annulus with an exceptional sphere; in particular,
there are no singular fibers if $(W,\omega)$ is minimal.}
\end{thm}

There is also a stability result for the Lefschetz fibrations considered
thus far.  Note that in the following, we don't assume the symplectic
forms $\omega_t$ are cohomologous.  \red{This result is applied in
\cite{Wendl:fiberSums} to classify strong fillings of various contact
manifolds up to symplectic deformation equivalence.}

\begin{thm}
\label{thm:stability}
If $(W,\omega_t)$ for $t \in [0,1]$ is a smooth
$1$--parameter family of strong fillings of either a planar contact manifold
$(M,\xi)$ or $(T^3,\xi_0)$, then by attaching a smooth family of 
trivial symplectic
cobordisms, one can construct a smooth family of strong fillings
$(W',\omega'_t)$ for which $\omega'_t$ is independent of $t$ near $\p W'$,
and there exists a smooth family of $\omega'_t$--symplectic
Lefschetz fibrations
$\Pi_t : W' \to \Sigma$ as in Theorems~\ref{thm:Lefschetz} and~\ref{thm:T3},
such that the critical points vary smoothly with~$t$.
\end{thm}

\subsection{Classifying \red{strong} fillings of $T^3$}

\red{Stipsicz showed using a gauge theory argument
\cite{Stipsicz:gaugeStein} that all Stein fillings of~$T^3$ are 
homeomorphic to $T^2 \times \DD$,}
and conjectured that this result can be strengthened
to a diffeomorphism.  In fact, more turns out to be true: \red{by
Theorem~\ref{thm:T3}, every minimal strong filling $W$ of $T^3$ admits a
symplectic fibration over the annulus with cylindrical fibers.}
One can now repeat this construction starting from a different
decomposition of $T^3$ (corresponding to a change in the 
$(q_1,q_2)$--coordinates), and thus show that $W$ admits two \red{symplectic
fibrations} over the annulus, with cylindrical fibers such that
any two fibers from each fibration intersect each other once transversely.
This provides a diffeomorphism from $W$ with an attached cylindrical end
to $T^*T^2$, and in \S\ref{sec:T3} 
we will use Moser isotopy arguments to show:

\begin{thm}
\label{thm:SteinT3}
\red{All minimal strong fillings of $T^3$ are symplectically deformation
equivalent, and every exact filling of $T^3$ 
is symplectomorphic to a star shaped domain
in $(T^*T^2,\omega_0)$.}
\end{thm}

\begin{cor}
Every \red{minimal strong} filling of $T^3$, \red{and in particular every
Stein filling}, is diffeomorphic to $T^2 \times \DD$.
\end{cor}

The first uniqueness result of this type was
obtained by Eliashberg \cite{Eliashberg:diskFilling}, who showed that all
Stein fillings of $S^3$ are diffeomorphic to the $4$--ball.  Shortly 
afterwards, McDuff \cite{McDuff:rationalRuled} classified Stein fillings
of the Lens spaces $L(p,1)$ with their standard contact structures up to
diffeomorphism, showing in particular that they are unique for all $p \ne 4$.
McDuff argued by compactification in order to apply her classification
results for rational and ruled symplectic $4$--manifolds, and several
other uniqueness and finiteness results have since been obtained using
similar ideas, e.g.~\cites{Lisca:fillingsLens,OhtaOno:simpleSingularities}.
\red{Many of these uniqueness results can be
recovered, and some of them strengthened or generalized, using the punctured 
holomorphic curve techniques introduced here (cf.~\cite{Wendl:fiberSums}).}
By contrast, there are also contact manifolds that admit infinitely many
non-diffeomorphic or non-homeomorphic Stein fillings: see
\cite{AkhmedovEtnyreMarkSmith} and the references mentioned therein.

The aforementioned result of McDuff for $L(p,1)$ was strengthened to
uniqueness up to Stein deformation equivalence by R.~Hind \cite{Hind:Lens},
using a construction similar to ours, though the technical arguments are
somewhat different.  Hind uses a foliation by $J$--holomorphic planes
asymptotic to a multiply covered orbit; since planes cannot undergo nodal 
degenerations unless there are closed curves involved, singular fibers are 
ruled out and the result is a smooth symplectic
fibration outside of the asymptotic orbit.  This fibration
can then be used to construct a plurisubharmonic function
with control over the critical points, thus leading to a uniqueness
result up to Stein homotopy.  It is plausible that one could apply 
Hind's idea to our construction
and further sharpen our classification of Stein fillings for $T^3$,
though we will not pursue this here.

Another consequence of Theorem~\ref{thm:SteinT3} (and also a step in its
proof) is that every \red{exact} filling of $T^3$ becomes symplectomorphic 
to $(T^*T^2,\omega_0)$ after attaching a positive cylindrical end.
It is then natural to ask about the topology of the compactly supported
symplectomorphism group.  In \S\ref{sec:T3} we will prove:

\begin{thm}
\label{thm:sympT3}
The group $\Symp_c(T^*T^2,\omega_0)$ of symplectomorphisms with compact
support is contractible.
\end{thm}

\subsection{Obstructions to fillability}

The results stated so far all 
start with the assumption that a filling exists, and
then use the existence of some $J$--holomorphic curves to deduce properties
of the filling.  In other situations, the same argument can sometimes lead
to a contradiction, thus defining an obstruction to filling---to 
understand this, we must first recall
some general notions about holomorphic curves in symplectizations
and finite energy foliations.  

If $\lambda$ is a contact form on $M$, then
the \emph{Reeb vector field} $X_\lambda \in \Vectors(M)$ is defined by the
conditions
$$
d\lambda(X_\lambda,\ ) \equiv 0,
\qquad
\lambda(X_\lambda) \equiv 1.
$$
The \emph{symplectization} $\RR \times M$ then admits a natural splitting of
its tangent bundle $T(\RR\times M) = \RR \oplus \RR X_\lambda \oplus \xi$;
let us denote the $\RR$--coordinate on $\RR\times M$ by $a$ and let
$\p_a$ denote the corresponding unit vector field.
There is now a nonempty and contractible space $\jJ_\lambda(M)$ of almost 
complex structures $J$ on $\RR\times M$ having the following properties:
\begin{itemize}
\item
$J$ is invariant under the $\RR$--action by translation on $\RR\times M$
\item
$J \p_a = X_\lambda$
\item
$J \xi = \xi$ and $J|_{\xi}$ is compatible with the symplectic structure
$d\lambda|_{\xi}$
\end{itemize}
Given $J \in \jJ_\lambda(M)$, we will consider $J$--holomorphic curves
$$
u : (\dot{\Sigma},j) \to (\RR\times M,J)
$$
where $(\Sigma,j)$ is a closed Riemann surface, $\dot{\Sigma} = \Sigma
\setminus\Gamma$ is the punctured surface determined by some finite subset
$\Gamma \subset\Sigma$, and $u$ has \emph{finite energy} in the sense
defined in \cite{Hofer:weinstein}.  The simplest examples of such curves
are the so-called \emph{orbit cylinders}
$$
\tilde{x} : \RR\times S^1 \to \RR\times M : (s,t) \mapsto (Ts,x(Tt)),
$$
for any $T$--periodic orbit $x : \RR \to M$ of $X_\lambda$.
We will not need to recall the precise
definition of the energy here, only that its finiteness constrains the
behavior of $u$ at the punctures: each puncture is either removable or
represents a positive/negative \emph{cylindrical end}, at which $u$ 
approximates an orbit cylinder,
asymptotically approaching a (perhaps multiply covered) periodic orbit
in $\{\pm\infty\} \times M$.

Recall that a $T$--periodic orbit is called \emph{nondegenerate} if the 
\red{transversal restriction of the} linearized time~$T$ flow along the orbit
does not have $1$ as an eigenvalue.
More generally, a \emph{Morse-Bott submanifold} of $T$--periodic orbits is a
submanifold $N \subset M$ consisting of $T$--periodic orbits such that the
$1$--eigenspace of the linearized flow is always precisely the tangent space
to~$N$.  We say that $\lambda$ is \emph{Morse-Bott} if every periodic orbit
belongs to a Morse-Bott submanifold; this will be a standing assumption 
throughout.  Note that a nondegenerate orbit is itself a ($1$--dimensional)
Morse-Bott submanifold.

Now consider a compact $3$--dimensional submanifold $M_0 \subset M$, possibly 
with boundary, such that $\p M_0$ is a Morse-Bott submanifold.  The following
objects were originally considered in \cite{HWZ:foliations}:

\begin{defn}
A \emph{finite energy foliation} $\fF$ on $(M_0,\lambda,J)$ is a foliation
of $\RR \times M_0$ with the following properties:
\begin{itemize}
\item
For any leaf $u \in \fF$, the $\RR$--translation of $u$ by any real number
is also a leaf in $\fF$.
\item
Every $u \in \fF$ is the image of an embedded finite energy $J$--holomorphic
curve satisfying a uniform energy bound.
\end{itemize}
\end{defn}
In light of the second requirement, we shall often blur the distinction between
leaves and the $J$--holomorphic curves that parametrize them.
The definition has several immediate consequences: most notably,
let $\pP_\fF$ denote the set of all simple periodic orbits that have covers
occurring as asymptotic
orbits for leaves of $\fF$.  Then an easy positivity of intersections argument
(see e.g.~\cite{Wendl:thesis}) implies that for each $\gamma \in \pP_\fF$,
the orbit cylinder $\RR\times \gamma$ is a leaf in $\fF$, and every leaf that
isn't one of these remains embedded under the natural projection
$$
\pi : \RR \times M \to M.
$$
In fact, abusing notation to regard $\pP_\fF$ as a subset
of $M$, the quotient $\fF / \RR$ defines a smooth foliation of $M_0 \setminus 
\pP_\fF$ by embedded surfaces transverse to $X_\lambda$.  These projected
leaves are noncompact and have closures with boundary in 
$\pP_\fF$.  It is easy to see from this that $\p M_0 \subset \pP_\fF$.

As we will see in Example~\ref{ex:torsion}, it is relatively easy to
construct finite energy foliations in various simple local models of contact
manifolds, and this will suffice for the obstruction to fillability that
we have in mind.  Global constructions are harder but do exist, for
instance on the tight $3$--sphere \cite{HWZ:foliations}, on overtwisted
contact manifolds \cite{Wendl:OTfol} and more generally on planar contact
manifolds \cites{Abbas:openbook,Wendl:openbook}.

\begin{defn}
We will say that a finite energy foliation $\fF$ on $(M_0,\lambda,J)$ is
\emph{positive} if every leaf that isn't an orbit cylinder has
only positive ends.  
\end{defn}
\begin{defn}
A leaf $u \in \fF$ will be called an \emph{interior}
leaf if it is not an orbit cylinder and all its ends belong to Morse-Bott 
submanifolds that lie in the
interior of $M_0$.  
\end{defn}
\begin{defn}
A leaf $u \in \fF$ will be called \emph{stable}
if it has genus~$0$, all its punctures are \emph{odd} and $\ind(u) = 2$
(see the appendix for the relevant technical definitions).
\end{defn}
This notion of a \emph{stable} leaf is meant to ensure that $u$
behaves well in the deformation and intersection theory of $J$--holomorphic
curves.  In practice, these
conditions are easy to achieve for leaves of genus zero.
\begin{defn}
A leaf $u \in \fF$ will be called \emph{asymptotically simple} if all its
asymptotic orbits are simply covered and belong to pairwise disjoint
Morse-Bott families; moreover every nontrivial Morse-Bott family among
these is a circle of orbits foliating a torus.
\end{defn}
\begin{remark}
\label{remark:asympSimple}
This last condition can very likely be relaxed, but it's satisfied
by most of the interesting examples I'm aware of so far and will simplify
the compactness argument in \S\ref{sec:compactness} considerably, particularly
in proving that limit curves are somewhere injective.
\end{remark}

\begin{thm}
\label{thm:obstruction}
Suppose $(M,\xi)$ has a Morse-Bott contact form $\lambda$, almost
complex structure
$J \in \jJ_\lambda(M)$ and compact $3$--dimensional 
submanifold $M_0$ with Morse-Bott boundary,
such that $(M_0,\lambda,J)$ admits a positive finite energy foliation $\fF$
containing an interior, stable and asymptotically simple leaf $u_0 \in \fF$.  
Assume also that either of the following is true:
\begin{enumerate}
\item $M_0 \subsetneq M$.
\item
There exists a leaf $u' \in \fF$ which is not an orbit cylinder and
is \emph{different} from some interior stable leaf $u_0$ in the following sense: 
either $u_0$ and
$u'$ are not diffeomorphic, or if they are, then there is no bijection between
the ends of $u_0$ and $u'$ such that the asymptotic orbits of $u_0$ are all
homotopic along Morse-Bott submanifolds to the corresponding asymptotic
orbits of $u'$.
\end{enumerate}
Then $(M,\xi)$ is not strongly fillable.
\end{thm}

The idea behind this obstruction is that if $(M,\xi)$ contains such a foliation
and is fillable, one can extend the foliation into the filling and derive a
contradiction by following the family of holomorphic 
curves along a path leading either outside of $M_0$ or to a ``different'' leaf
$u' \in \fF$.  As we'll note in Remark~\ref{remark:Weinstein}, a similar
argument leads to a proof of the Weinstein conjecture whenever a subset of
$M$ admits a finite energy foliation with an interior, stable 
and asymptotically simple leaf.

\begin{example}[Overtwisted contact structures]
It was shown in \cite{Wendl:OTfol} that every overtwisted contact
manifold globally admits a finite energy foliation satisfying the
conditions of Theorem~\ref{thm:obstruction}, so this implies a new 
\red{(admittedly much harder)} proof of
the classic Eliashberg-Gromov result that all strongly fillable contact 
structures are tight (see also Remark~\ref{remark:weakly}).  
The foliation in question is produced by starting from a
planar open book decomposition in $S^3$ and performing Dehn surgery and
Lutz twists along a transverse link: each component of the link is surrounded
by a torus which becomes a Morse-Bott submanifold in the foliation
(see Figure~\ref{fig:overtwisted}).  Note that an easier proof that
strongly fillable manifolds are tight 
is possible using the result for Giroux torsion
below; cf.~\cite{Gay:GirouxTorsion}*{Corollary~5}.
\end{example}

\begin{figure}
\includegraphics{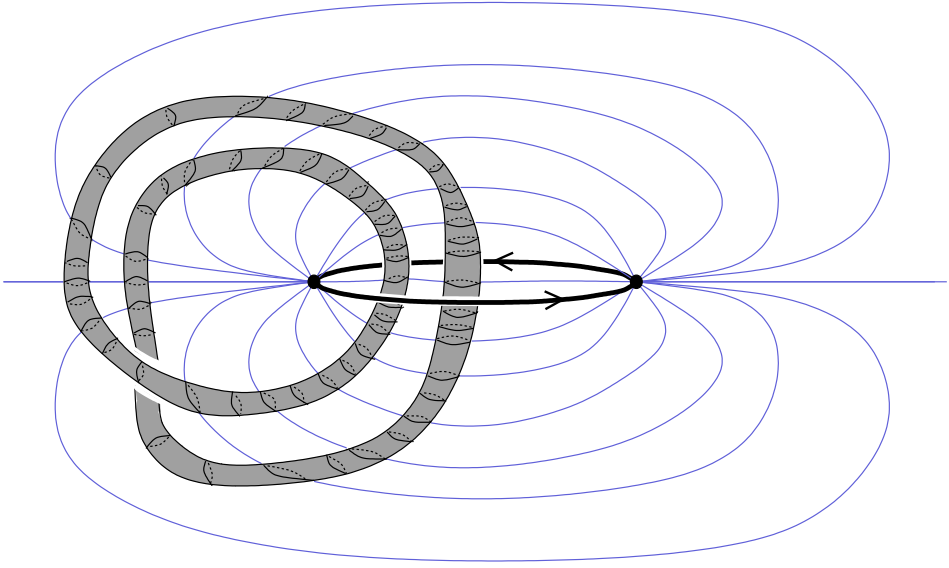}
\caption{\label{fig:overtwisted} 
A global finite energy foliation produced from a planar open book decomposition
on $S^3$ by surgery along a transverse link.  Any overtwisted contact manifold
can be foliated this way, giving a new proof that strongly fillable contact
manifolds are tight.}
\end{figure}

\begin{example}[Giroux torsion]
\label{ex:torsion}
Let $T^2 = S^1 \times S^1$
and $T = T^2 \times [0,1]$ with coordinates $(q_1,{q_2},\theta)$.
Given smooth functions $f, g : [0,1] \to \RR$, a $1$--form
\begin{equation*}
\label{eqn:lambdaGeneral}
\lambda = f(\theta)\ dq_1 + g(\theta)\ d{q_2}
\end{equation*}
is a positive contact form 
if and only if $D(\theta) := f(\theta) g'(\theta) - f'(\theta) g(\theta) > 0$,
meaning the path $\theta \mapsto (f,g) \in \RR^2$ winds counterclockwise 
around the origin.  An important special case is the $1$--form
$$
\lambda_1 = \cos( 2\pi \theta)\ dq_1 + \sin(2\pi \theta)\ d{q_2},
$$
with contact structure $\xi_1 := \ker\lambda_1$.  A closed contact manifold
$(M,\xi)$ is said to have \emph{positive Giroux torsion} if it admits a contact
embedding of $(T,\xi_1)$.  Recently, D.~Gay \cite{Gay:GirouxTorsion} used
gauge theory to show that contact manifolds with positive Giroux torsion
are not strongly fillable, and another proof using the Ozsv\'{a}th-Szab\'{o}
contact invariant has been carried out by Ghiggini, Honda and
Van Horn-Morris \cite{GhigginiHondaVanhorn}.  We shall now reprove this
result by constructing an appropriate finite energy foliation in $T$;
a pictorial representation of the proof is shown in 
Figure~\ref{fig:nonfillable}.

\begin{figure}
\includegraphics{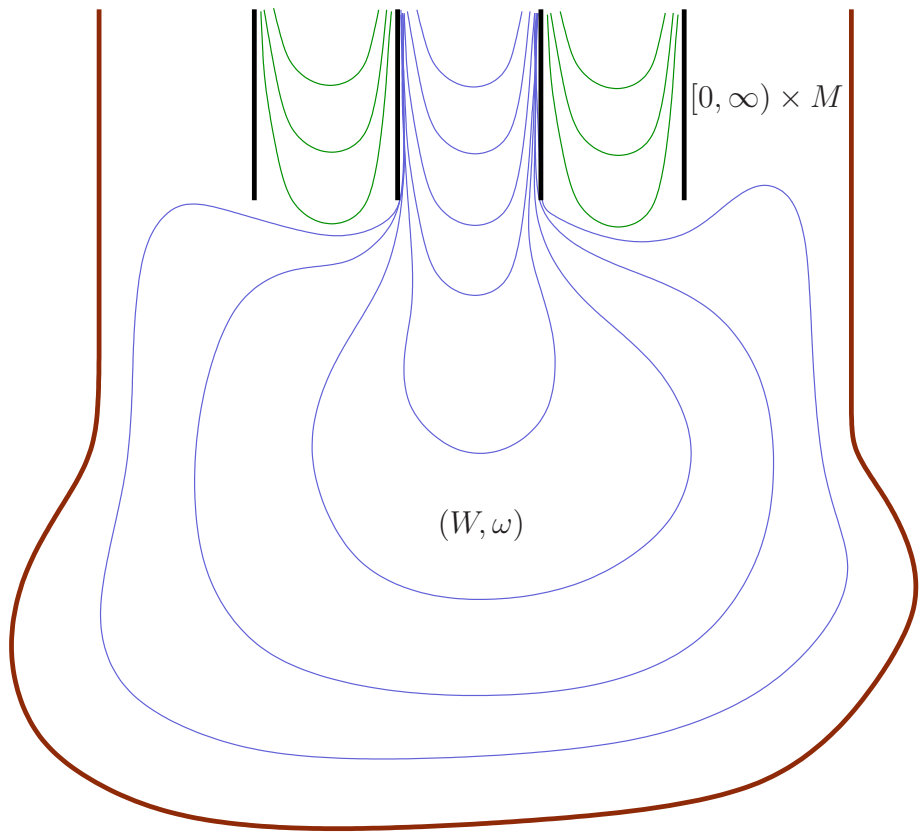}
\caption{\label{fig:nonfillable} 
The reason why Giroux torsion contradicts strong fillability: one can
construct a finite energy foliation consisting of three families of holomorphic
cylinders with positive ends.  The middle family contains interior stable
leaves, which then spread to a foliation of any filling and must
eventually run into the other families, giving a contradiction.}
\end{figure}

First note that one can always slightly expand the embedding of $T$ and thus
replace it with $T' := T^2 \times [-\epsilon,1 + \epsilon]$ for
some small $\epsilon > 0$, with the same contact form $\lambda_1$ as above.
Now multiplying the contact form by a smooth positive function of $\theta$,
we can replace $\lambda_1$ by $\lambda = f(\theta)\ dq_1 + g(\theta)\ d{q_2}$
such that $g'(-\epsilon) = g'(1 + \epsilon) = 0$.  Note that also
$g'(1/4) = g'(3/4) = 0$.  The result is that these four special values of
$\theta$ all define Morse-Bott tori foliated by closed Reeb orbits in the
$\pm\p_{q_2}$ direction (with signs alternating).  Indeed, it is easy to
compute that the Reeb vector field takes the form
$$
X_\lambda(q_1,{q_2},\theta) = \frac{g'(\theta)}{D(\theta)} \p_{q_1} -
\frac{f'(\theta)}{D(\theta)} \p_{q_2}.
$$
Now choose $J$ to be a complex structure on $\xi_1$ such that
$$
J (C \p_\theta) = - \frac{g(\theta)}{D(\theta)} \p_{q_1} 
+ \frac{f(\theta)}{D(\theta)} \p_{q_2}
$$
for some constant $C > 0$.
As shown in \cite{Wendl:OTfol}*{\S 4.2}, it is easy to construct a foliation
by holomorphic cylinders in this setting: we simply suppose there exist
cylinders $u : \RR \times S^1 \to \RR \times T'$ of the form
$$
u(s,t) = (a(s),c,t,\theta(s)),
$$
where $c \in S^1$ is a constant,
and find that the nonlinear Cauchy-Riemann equations reduce to a pair of
ODEs for $a(s)$ and $\theta(s)$; these have unique global solutions for any
choice of $a_0 := a(0)$ and $\theta_0 := \theta(0)$.  In particular, the
solution $\theta(s)$ is monotone and maps $\RR$ bijectively onto the largest
interval $(\theta_-,\theta_+) \subset (-\epsilon,1+\epsilon)$ containing
$\theta_0$ on which $g'$ is nonvanishing.  Likewise, $a(s) \to +\infty$ as
$s \to \pm\infty$.  As a result, in each of the subsets
$\{ \theta \in (-\epsilon,1/4) \}$, $\{ \theta \in (1/4,3/4) \}$ and
$\{ \theta \in (3/4,1+\epsilon) \}$, we obtain a smooth 
$(\RR\times S^1)$--parametrized family of $J$--holomorphic curves that
foliate the corresponding region; adding in the trivial cylinders for all
four of the aforementioned Morse-Bott tori yields a positive finite energy
foliation of $T'$.  It is straightforward to verify that all curves in
the foliation are stable in the sense defined here.  Since the leaves in
$\{ \theta \in (1/4,3/4) \}$ have their asymptotic orbits in the
interior of $T'$, and all other leaves have asymptotic orbits on
different Morse-Bott submanifolds, Theorem~\ref{thm:obstruction} applies,
giving a completely non-gauge-theoretic proof that no contact manifold
containing $(T',\xi_1)$ can be strongly fillable.
\end{example}

\begin{remark}
\label{remark:weakly}
Giroux torsion is not generally an obstruction to \emph{weak} fillability,
e.g.~this was demonstrated with examples on $T^3$ by Giroux
\cite{Giroux:plusOuMoins} and Eliashberg \cite{Eliashberg:fillableTorus}.
Note also that overtwisted contact manifolds are not
weakly fillable, but our method \red{does not prove this}, as 
Theorem~\ref{thm:compactness} below requires the attachment of a positive
cylindrical end to the boundary of the filling.  This is an \red{important}
difference between our technique and the ``disk filling'' methods used by
Eliashberg in \cite{Eliashberg:diskFilling}.
\end{remark}

\begin{remark}
\label{remark:foliationT3}
The setup used in Example~\ref{ex:torsion} above for 
Giroux torsion is also suitable for $(T^3,\xi_0)$, thus the same
trick yields a positive stable finite energy foliation whose leaves project
to the fibers of the fibration \eqref{eqn:piT3}.  We will make 
use of this foliation in the proof of Theorem~\ref{thm:T3}.
\end{remark}

\begin{example}
\label{ex:connected}
We've generally assumed the contact manifold $(M,\xi)$ to be connected,
but one can also drop this assumption.  Theorem~\ref{thm:obstruction}
then applies, for instance, to any disjoint union of contact manifolds
containing a planar component.  One recovers in this way a result of 
Etnyre \cite{Etnyre:planar}, that any strong symplectic filling with a
planar boundary component must have connected boundary.  This applies
more generally if any boundary component admits a positive stable
finite energy foliation, e.g.~the standard~$T^3$.  \red{A further 
generalization to \emph{partially planar} contact manifolds
is explained in \cite{AlbersBramhamWendl}, using similar ideas.}
\end{example}

\section{Holomorphic curves and compactness}
\label{sec:compactness}

The theorems of the previous section are consequences of the compactness 
properties of pseudoholomorphic curves belonging to a foliation in a symplectic
$4$--manifold with a positive cylindrical end.
The setup for most of this section will be as 
follows: assume $(M,\xi)$ has a Morse-Bott
contact form $\lambda$
and almost complex structure $J_+ \in \jJ_\lambda(M)$, 
a compact $3$--dimensional
submanifold $M_0 \subset M$ with Morse-Bott boundary
and a positive finite energy foliation $\fF_+$ of $(M_0,\lambda,J_+)$ 
containing an interior stable leaf that is asymptotically 
simple.  Assume further that $(W^\infty,\omega)$
is a noncompact symplectic manifold admitting a decomposition
$$
W^\infty = W \cup_{\p W} \left([R,\infty) \times M \right)
$$
for some $R \in \RR$, where $W$ is a compact manifold with boundary 
$\p W = M$ and
$\omega|_{[R,\infty) \times M } = d(e^a\lambda)$, with $a$ denoting
the $\RR$--coordinate on $\RR \times M$.  
There is a natural compactification $\overline{W}^\infty$ of $W^\infty$,
defined by choosing any smooth structure on $[R,\infty]$ and replacing 
$[R,\infty) \times M$ in the above decomposition by $[R,\infty] \times M$;
then $\overline{W}^\infty$ is a compact smooth manifold with boundary
$\p\overline{W}^\infty = M$.

The open manifold $(W^\infty,\omega)$ is a natural setting for punctured 
pseudoholomorphic curves.  Indeed, choose any number
$$
a_0 \in [R,\infty)
$$
and an almost complex structure $J$ on $W^\infty$ that is compatible with
$\omega$ and satisfies $J|_{[a_0,\infty) \times M} = J_+$.
Just as in the symplectization $\RR\times M$, one then considers
punctured $J$--holomorphic curves of finite energy in $W^\infty$, 
such that each 
puncture is a positive end approaching a Reeb orbit at $\{+\infty\} \times M$.

Let $\fF_0$ denote the collection of leaves in $\fF_+$ that lie entirely
within $[a_0,\infty) \times M$: observe that this includes some 
$\RR$--translation of every leaf that isn't an orbit cylinder.  Then
each of these leaves embeds naturally into $W^\infty$ as a finite energy
$J$--holomorphic curve.  After a generic perturbation of $J$ compatible
with $\omega$ in the region
$W \cup \left((R,a_0)\times M\right)$, 
standard transversality arguments as in 
\cite{McDuffSalamon:Jhol} imply that every somewhere injective 
$J$--holomorphic curve $v : \dot{\Sigma} \to W^\infty$ not fully contained in
$[a_0,\infty) \times M$ satisfies $\ind(v) \ge 0$.  
We will assume $J$ satisfies
this genericity condition unless otherwise noted.

\begin{remark}
Note that we are \emph{not} assuming $J_+ \in \jJ_\lambda(M)$ is generic,
which is important because we wish to apply the results below for
foliations $(M_0,\lambda,J_+)$ as constructed in Example~\ref{ex:torsion},
where $J_+$ is chosen to be as symmetric as possible.  We can get away with
this because of the distinctly $4$--dimensional phenomenon of 
``automatic'' transversality: in particular,
Prop.~\ref{prop:automatic} guarantees transversality for stable
leaves without any genericity assumption.  We need genericity in the
compactness argument of Theorem~\ref{thm:compactness} only to ensure that
nodal curves with components of negative index do not appear.
\end{remark}

Denote by $\mM$ the moduli space of finite energy 
$J$--holomorphic curves in~$W^\infty$, and
let $\overline{\mM}$ denote its natural compactification as in
\cite{SFTcompactness}: the latter consists of \emph{nodal $J$--holomorphic 
buildings},
possibly with multiple levels, including \red{a \emph{main} level}
in $W^\infty$ and
several \emph{upper} levels, \red{which are equivalence classes of nodal curves
in $\RR\times M$ up to $\RR$--translation}.  There are no \emph{lower} 
levels since $W^\infty$ has no negative end.

Choose any interior stable leaf 
$u_0 \in \fF_0$ that is asymptotically simple,
let $\mM_0 \subset \mM$ be the connected component containing $u_0$
and $\overline{\mM}_0 \subset \overline{\mM}$ the closure of $\mM_0$.

We will now prove two compactness results: one that gives the existence of a
global foliation with isolated singularities on $W^\infty$, and another 
that preserves this foliation under generic homotopies of the data.

\begin{thm}
\label{thm:compactness}
If $M$ contains a submanifold $M_0$ with finite energy foliation $\fF_+$
as described above, then $M_0 = M$.  Moreover,
the moduli spaces $\mM_0$ and $\overline{\mM}_0$ have the following properties:
\begin{enumerate}
\item
Every curve in $\mM_0$ is embedded and unobstructed (i.e.~the linearized
Cauchy-Riemann operator is surjective), 
and no two curves in $\mM_0$ intersect.
\item
$\overline{\mM}_0 \setminus \mM_0$ consists of the following:
\begin{enumerate}
\item
A compact $1$--dimensional manifold of buildings that each have an empty
\red{main} level and one nontrivial upper level that is a leaf of $\fF_+$
\red{(see Remark~\ref{remark:Rtranslation} below)},
\item
A finite set of $1$--level
nodal curves in $W^\infty$, each consisting of two embedded index~$0$
components with self-intersection number~$-1$ (see Remark~\ref{remark:selfint}
below), which
intersect each other exactly once, transversely.  These are all disjoint
from each other and from the smooth embedded curves in $\mM_0$.
\end{enumerate}
\item
The collection of curves in $\mM_0$ plus the embedded curves in $W^\infty$ that 
form components of nodal curves in $\overline{\mM}_0$ forms a foliation of $W^\infty$ 
outside of a finite set of ``double points'' where two leaves intersect 
transversely; these are the nodes of the isolated
nodal curves in $\overline{\mM}_0 \setminus \mM_0$.
\item
$\overline{\mM}_0$ is a smooth manifold diffeomorphic to either
$[0,1] \times S^1$ or $\DD$; it is the latter if and only if every
asymptotic orbit of the interior stable leaf $u_0$ is nondegenerate.
\end{enumerate}
\end{thm}

\red{
\begin{remark}
\label{remark:Rtranslation}
Note that the curves in the upper levels of a building are technically only
\emph{equivalence classes} of curves up to $\RR$--translation, nonetheless
it makes sense to speak of such a curve being a leaf of $\fF_+$, since the
latter is also an $\RR$--invariant foliation.
\end{remark}
}

\begin{remark}
\label{remark:selfint}
The self-intersection number here is meant to be interpreted in the sense
of Siefring's intersection theory for punctured holomorphic curves
\cites{Siefring:intersection,SiefringWendl}.  This is reviewed briefly in the
appendix, though it's most important to consider the case where the
curve under consideration is closed: then the definition of
``self-intersection number'' reduces to the usual one.
\end{remark}

\begin{proof}
As preparation, note that the stability condition for $u_0$ implies due to
\eqref{eqn:cNindex} that its normal Chern number $c_N(u_0)$ vanishes, 
hence $2 = \ind(u) > c_N(u) = 0$
for all $u \in \mM_0$.  The transversality criterion of 
Prop.~\ref{prop:automatic} thus guarantees that every $u \in \mM_0$ is
unobstructed once we prove that it is also embedded; we will do this
in Step~7.  The proof now proceeds in several steps.

\textbf{Step~1}: We claim that no curve $u \in \mM_0$
can have an isolated intersection with any leaf
$u_+ \in \fF_0$.  Clearly, for any given $u_+ \in \fF_0$,
positivity of intersections implies that
the subset of curves $u \in \mM_0$ that have no isolated intersection with
$u_+$ is closed, and we must show that it's also open.  There's a slightly
subtle point here, as the noncompactness of the domain allows a theoretical
possibility for intersections to ``emerge from infinity'' under perturbations
of~$u$.  To rule this out, we use the intersection theory of punctured
holomorphic curves defined in \cites{Siefring:intersection,SiefringWendl} 
(a basic outline is given in the appendix).  The point is that there exists 
a homotopy invariant intersection number $i(u ; u_+) \in \ZZ$ 
that includes a count of ``asymptotic intersections'', and
the condition $i(u ; u_+) = 0$ is sufficient to guarantee that no curve
homotopic to $u$ ever has an isolated intersection with $u_+$.  This
number vanishes in the present case due to Lemma~\ref{lemma:i0}.

\textbf{Step~2}: As an obvious consequence of Step~1, a similar statement is 
true for any component $v$ of a building $u \in \overline{\mM}_0$: 
$v$ has no isolated intersection with any leaf
$u_+ \in \fF_+$ if $v$ is in an upper level, or with any $u_+ \in \fF_0$
if $v$ is in the \red{main} level.

\textbf{Step~3}: If $u \in \overline{\mM}_0 \setminus \mM_0$, 
we claim that one of the following is true:
\begin{enumerate}
\item
$u$ has only one nontrivial upper level, consisting of a leaf of $\fF_+$ in
$\RR\times M$, and the \red{main} level is empty.
\item
$u$ has \red{no upper levels}.
\end{enumerate}
Indeed, suppose $u$ has nontrivial upper levels and let $v$ denote a
nontrivial component of the topmost nontrivial level.  Due to our
assumptions on $u_0$, each positive end of $v$ is then a simply covered
orbit belonging to a distinct Morse-Bott submanifold in the interior of $M_0$, 
hence $v$ is somewhere injective.  The asymptotic formula of 
\cite{HWZ:props4} now implies that $\pi \circ v$ is an embedding into $M$ 
near each
end and is disjoint from the corresponding asymptotic orbit; hence it
intersects some projected leaf of $\fF_+$; we conclude that $v$ intersects 
some leaf $u_+ \in \fF_+$.  By the result of Step~2, this intersection cannot
be isolated, and since $v$ is somewhere injective, we conclude $v \in \fF_+$.
As a result, $v$ has no negative ends and its positive ends are in 
one-to-one correspondence with those of $u_0$, so $u$ can have no other
nonempty components.

\textbf{Step~4}: Suppose $u \in \overline{\mM}_0 \setminus \mM_0$ 
satisfies the second alternative in Step~3: $u$ is
then a nodal curve in the \red{main} level.  We claim that any nonconstant
component $v$
of $u$ either is a leaf in $\fF_0$ or it is not contained in the subset
$[a_0,\infty) \times M \subset W^\infty$.  There are two cases to consider: if
$v$ has no ends then it cannot be in $[a_0,\infty) \times M$ because the
symplectic form here is exact, so no nonconstant closed holomorphic curve
can exist.  If on the other hand $v$ has positive ends and is contained in
$[a_0,\infty) \times M$, where $J$ is $\RR$--invariant, then a similar
argument as in Step~3 finds an illegal isolated intersection of $v$ with
a leaf of $\fF_0$ unless $v$ is such a leaf.

\textbf{Step~5}: Continuing with the assumptions of Step~4,
we claim that one of the following holds:
\begin{enumerate}
\item
$u$ is smooth (i.e.~has no nodes).
\item
$u$ has exactly two components, both somewhere injective and with index~$0$.
\end{enumerate}
To see this, recall first that $u_0$ has genus~$0$, thus $u$ has
arithmetic genus~$0$.  
Now suppose $u$ has multiple components connected by $N \ge 1$ nodes.  Every 
component of $u$ is then either a punctured sphere with 
positive ends (denoted
here by $v_i$), a nonconstant closed sphere (denoted $w_i$) or a
\emph{ghost bubble}, i.e.~a constant sphere (denoted $g_i$).  For a sphere 
$v_i$ with ends, the asymptotic behavior of $u_0$
guarantees that $v_i$ is somewhere injective.  Then by Step~4, it is either
a leaf of $\fF_0$ or it is not contained in
$[a_0,\infty) \times M$, hence the genericity assumption for $J$ implies
$\ind(v_i) \ge 0$.  Consider now a nonconstant closed component $w_i$, which we assume
to be a $k_i$--fold cover of a somewhere injective sphere $\hat{w}_i$ for
some $k_i \in \NN$.  Again, Step~4 and the genericity of $J$ guarantee that
$\ind(\hat{w}_i) = 2 c_1([\hat{w}_i]) - 2 \ge 0$, hence
$$
\ind(w_i) = 2 c_1([w_i]) - 2 = 2 k_i c_1([\hat{w}_i]) - 2 =
k_i \cdot \ind(\hat{w}_i) + 2(k_i - 1) \ge 2(k_i - 1).
$$
Ghost bubbles are now easy to rule out: we have $\ind(g_i) = 2 c_1([g_i]) - 2
= -2$, and by the stability condition of Kontsevich (cf.~\cite{SFTcompactness}),
$g_i$ has at least three nodes, each contributing~$2$ to the total index
of $u$.  Since we already know that the nonconstant components contribute
nonnegatively to the index, the existence of a ghost bubble thus implies
the contradiction $\ind(u) \ge 4$.  With this detail out of the way, we 
add up the
indices of all components, counting an additional~$2$ for each node, and find
\begin{equation*}
\begin{split}
2 = \ind(u) &= \sum_i \ind(v_i) + \sum_i \ind(w_i) + 2N \\
&\ge 2\sum_i (k_i - 1) + 2N.
\end{split}
\end{equation*}
Since $N \ge 1$ by assumption, this implies that each $k_i$ is~$1$ and
$N=1$, hence $u$ has exactly two components, both somewhere injective with
index~$0$.

\textbf{Step~6}: By Step~5, the nodal curves in $\mM_0$
have components that are unobstructed and have index~$0$, 
hence they are isolated.  By the compactness of $\overline{\mM}_0$,
this implies that the set of nodal
curves in $\overline{\mM}_0 \setminus \mM_0$ is finite.  A standard
gluing argument as in \cite{McDuffSalamon:Jhol} now identifies a neighborhood
of any nodal curve $u$ in $\overline{\mM}_0$ with an open subset of $\RR^2$,
where every curve other than $u$ is smooth.  Similarly, since every
$u \in \mM_0$ is unobstructed, the usual implicit function theorem in
Banach spaces defines smooth manifold charts everywhere on $\mM_0$.  Outside
a compact subset, $\overline{\mM}_0 \setminus \p\overline{\mM}_0$ can be
identified with the set of leaves in $\fF_0$, and is thus diffeomorphic
to $[0,\infty) \times V$ for some compact $1$--manifold $V$, so
$\p\overline{\mM}_0$ is diffeomorphic to $V$ itself.
The space $\overline{\mM}_0$ is therefore a compact surface with boundary,
and is orientable due to arguments in \cite{BourgeoisMohnke}.

\textbf{Step~7}: We now use the intersection theory from
\cites{Siefring:intersection,SiefringWendl} to show that
$\overline{\mM}_0$ foliates~$W^\infty$.  We noted already in Step~1 that
$i(u ; u') = 0$ for any two curves $u, u' \in \overline{\mM}_0$, which
implies that no two of these curves can ever intersect.  Since every
$u \in \mM_0$ is obviously somewhere injective due to its asymptotic
behavior, the adjunction formula \eqref{eqn:adjunction} implies
$\sing(u) = 0$ and thus these curves are also embedded.  Consider now
a nodal curve $u \in \overline{\mM}_0$, with its two components $u_1$ and~$u_2$,
and observe that \eqref{eqn:cNindex} implies $c_N(u_1) = c_N(u_2) = -1$.
Applying the adjunction formula again, we find
\begin{equation*}
\begin{split}
0 = i(u ; u) &= i(u_1 ; u_1) + i(u_2 ; u_2) + 2 i(u_1 ; u_2) \\
&\ge 2\sing(u_1) + c_N(u_1) + 2\sing(u_2) + c_N(u_2) + 2 i(u_1 ; u_2) \\
&= 2\sing(u_1) + 2\sing(u_2) + 2 \left[ i(u_1 ; u_2) - 1 \right].
\end{split}
\end{equation*}
Thus $\sing(u_1) = \sing(u_2) = 0$, implying both components are embedded,
and $i(u_1 ; u_2) = 1$, so the node is the only intersection, and is
transverse.  The adjunction formula for each of $u_1$ and $u_2$ individually
now also implies $i(u_1 ; u_1) = i(u_2 ; u_2) = -1$.  (Note that the
$\cov_\infty(z)$ terms must all vanish, as this is manifestly true
for $u_0$ and they depend only on the orbits).
By the gluing argument mentioned in Step~6,
a neighborhood of $u$ in $\overline{\mM}_0$ is a smooth
$2$--parameter family of embedded curves from $\mM_0$; these foliate a
neighborhood of the union of $u_1$ and $u_2$.  Similarly, the implicit
function theorem in \cite{Wendl:BP1} or \cite{Wendl:thesis} implies that
for any $u \in \mM_0$, the nearby curves in $\mM_0$ foliate a neighborhood
of~$u$.  This shows that
$$
\{ p \in W^\infty\ |\ \text{$p$ is in the image of some $u \in
\overline{\mM}_0$} \}
$$
is an open subset of~$W^\infty$.  It is also clearly a closed subset since
$\overline{\mM}_0$ is compact.  We conclude that all of $W^\infty$ is filled
by the curves in $\overline{\mM}_0$.

\textbf{Step~8}: It follows easily now that $M_0 = M$, as one can take a
sequence of curves in $\mM_0$ whose images approach $(+\infty,p)$ for
any $p \in M$; since a subsequence converges to a leaf of $\fF_+$,
we conclude that $\fF_+$ fills all of~$M$.

\textbf{Step~9}: Having shown already that $\overline{\mM}_0$ is a compact
orientable surface with boundary, we prove finally that it must be either
$\DD$ or $[0,1] \times S^1$.  Define a smooth map
\begin{equation}
\label{eqn:Pi}
\Pi : W^\infty \to \overline{\mM}_0
\end{equation}
by sending $p \in W^\infty$ to the unique curve in $\overline{\mM}_0$
whose image contains~$p$.  We can extend $\Pi$ over $\overline{W}^\infty
\setminus \pP_{\fF_+}$ by sending
$p \in M \setminus \pP_{\fF_+}$ to the unique leaf in $\fF_+ / \RR =
\p\overline{\mM}_0$ containing~$p$.

Assume first that there are degenerate
orbits among the asymptotic orbits of the interior stable leaf
$u_0 \in \fF_+$: such an orbit belongs to a
Morse-Bott $2$--torus $T_0 \subset M$ foliated by Reeb orbits that are
asymptotic limits of leaves in $\fF_+$.  By the definition of $\mM_0$,
every curve $u \in \mM_0$ and thus every leaf in $\fF_+$ has a unique end 
asymptotic to some orbit in~$T_0$.  In this case $\p\overline{\mM}_0$ must
have two connected components, and we can parametrize them as follows.
Identify a neighborhood of $T_0$ in $M$ with $(-1,1) \times S^1 \times S^1$
such that $\{0\} \times S^1 \times S^1 = T_0$ and the Reeb orbits are all
of the form $\{0\} \times \{\text{const}\} \times S^1$.  Then we can arrange
that for sufficiently small $\epsilon > 0$, 
the loop $\gamma_+(t) = (+\infty,\epsilon,t,0) \in \overline{W}^\infty$ passes 
through a different leaf of
$\fF_+$ for each~$t$, thus without loss of generality,
$\Pi \circ \gamma_+ : S^1 \to \p\overline{\mM}_0$ is an oriented parametrization
of one boundary component of $\p\overline{\mM}_0$.  The other boundary 
component can be given an oriented parametrization in the form
$\Pi \circ \gamma_- : S^1 \to \p\overline{\mM}_0$ where
$\gamma_-(t) = (+\infty,-\epsilon,-t,0)$.  
Now moving both loops down slightly from~$\infty$, we see that
$[\gamma_-] = -[\gamma_+] \in \pi_1(\overline{W}^\infty \setminus \pP_{\fF_+})$, 
implying that
the two boundary components of $\overline{\mM}_0$ are homotopic,
and therefore $\overline{\mM}_0 \cong [0,1] \times S^1$.

If all orbits of $u_0$ are nondegenerate, then $\p\overline{\mM}_0$ must have
only one component, which we can similarly parametrize by choosing a loop
$\gamma : S^1 \to \{+\infty\} \times M$ that circles once around one of these
orbits and passes once transversely through each leaf of~$\fF_+$.
Moving $\gamma$ again down from $+\infty$, it is contractible in
$\overline{W}^\infty \setminus \pP_{\fF_+}$, implying $\p\overline{\mM}_0$ is
contractible, thus $\overline{\mM}_0 \cong \DD$.
\end{proof}

To set up the second compactness result, assume that for $\tau \in [0,1]$,
$\omega_\tau$ is a smooth family of symplectic forms on $W^\infty$
matching $d(e^a\lambda)$ on $[a_0,\infty) \times M$, and $J_\tau$ is a smooth
family of almost complex structures compatible with $\omega_\tau$ for each
$\tau$ and matching $J_+ \in \jJ_\lambda(M)$ on $[a_0,\infty) \times M$.  
Assume also
that the homotopy $J_\tau$ is generic on $W^\infty \setminus \left(
[a_0,\infty) \times M \right)$ so that for any $\tau \in [0,1]$,
every somewhere injective $J_\tau$--holomorphic curve $u$ not contained in
$[a_0,\infty)\times M$ satisfies $\ind(u) \ge -1$.  Then for each $\tau$,
let $\mM_\tau$ denote the connected moduli space of $J_\tau$--holomorphic
curves containing an interior stable leaf in $\fF_0$ that is asymptotically
simple, and write its compactification as $\overline{\mM}_\tau$.

\begin{thm}
\label{thm:compactnessHomotopy}
The conclusions of Theorem~\ref{thm:compactness} hold for the moduli 
spaces $\overline{\mM}_\tau$ for each $\tau \in [0,1]$; 
in particular they are all smooth compact manifolds
with boundary that form foliations of $W^\infty$ with finitely many
singularities, and their boundaries can be identified naturally with the
set of leaves in the projected foliation $\fF_+ / \RR$.  Moreover,
there exists a smooth $1$--parameter family of
diffeomorphisms $\overline{\mM}_0 \to \overline{\mM}_\tau$
that maps $\mM_0$ to $\mM_\tau$ and restricts to the natural identification
$\p\overline{\mM}_0 \to \p\overline{\mM}_\tau$.
\end{thm}
\begin{proof}
For each $\tau \in [0,1]$, the proof of Theorem~\ref{thm:compactness}
requires only a small modification to work for the almost complex
structure~$J_\tau$.  The difference is that $J_\tau$ is now not necessarily
generic, so we have a weaker lower bound on the indices of somewhere
injective curves that are not contained in $[a_0,\infty) \times M$.  The only
place this makes a difference is in Step~5: we must now consider
the possibility that $u$ is a nodal curve in $W^\infty$ with several
components of possibly negative index.  Since none of these components
are contained in $[a_0,\infty) \times M$ and $\{ J_\tau \}_{\tau \in [0,1]}$
is a generic homotopy, they all cover somewhere injective curves of 
index at least~$-1$.  We claim that
this implies the somewhere injective curves have nonnegative index after all: 
for closed components
the index is always even, so this is clear.  The same turns out to be
true for components with ends: since $u_0$ has only odd punctures,
any punctured somewhere injective curve with a cover that forms a component
of $u$ has all its ends asymptotic to orbits that
have odd covers, and must themselves therefore be odd.  (See
\cite{Wendl:automatic}*{\S 4.2} for the proof that even orbits always
have even covers; this statement applies equally well in the Morse-Bott
setup described in the appendix.)  
It follows then from the index formula
that the index of such a component must be even, and in this case
therefore nonnegative.  The rest of the compactness proof now follows just 
as before, with the added detail that all curves arising in the limit 
(including components of nodal curves) are unobstructed due to 
Prop.~\ref{prop:automatic}, which does not require genericity.

By the above argument, we have moduli spaces $\overline{\mM}_\tau$ that
foliate $W^\infty$ with $J_\tau$--holomorphic curves outside of a finite
set of nodes.  Moreover, every curve in the foliation is unobstructed,
so for any given $\tau_0 \in [0,1]$, the index~$0$ curves that are components
of nodal curves in $\overline{\mM}_{\tau_0}$ deform uniquely to
$J_\tau$--holomorphic curves for $\tau$ in some neighborhood of $\tau_0$,
and an intersecting pair of such curves forms a nodal curve.  Since the
curves in $\overline{\mM}_{\tau_0}$ and $\overline{\mM}_\tau$ near their 
respective boundaries are identical, a familiar
intersection argument now shows that this nodal curve must belong to
$\overline{\mM}_\tau$.  Similarly, index~$2$ curves in 
$\mM_{\tau_0}$ deform to index~$2$ curves in $\mM_{\tau}$, providing a
local smooth $1$--parameter family of diffeomorphisms
$$
\overline{\mM}_{\tau_0} \to \overline{\mM}_\tau
$$
for $\tau$ close to $\tau_0$,
which maps nodal curves to nodal curves and leaves in $\fF_0$ and
$\fF_+$ to themselves.  To extend this for all $\tau \in [0,1]$, it only
remains to show that the ``parametrized'' moduli space
$$
\overline{\mM}_{[0,1]} := \{ (\tau,u)\ |\ 
\text{$\tau \in [0,1]$, $u \in \overline{\mM}_\tau$} \}
$$
is compact.  This follows from the same arguments as above, after
observing that the energies of $u \in \mM_\tau$ depend only on the
relative homology class defined by a leaf $u_0 \in \fF_0$ and
(continuously) on $\omega_\tau$, thus they are uniformly bounded.
\end{proof}

\begin{remark}
\label{remark:noGenericity}
In some important situations, one can prove the two theorems above without
any genericity assumption at all: the point is that genericity is usually 
needed to ensure a lower bound on the indices of components in nodal curves, but
is not required to show that the curves actually obtained in the limit are 
unobstructed.
Thus if there are topological conditions preventing the appearance of nodal
curves, then \emph{any} compatible $J$ or smooth family $J_\tau$ 
\red{(also for $\tau$ varying in a higher-dimensional space)} will suffice: 
this works in particular for \red{exact} fillings of $T^3$ and will play a 
crucial role in the proof of Theorem~\ref{thm:sympT3}.
\end{remark}

\section{Lefschetz fibrations and obstructions to filling}
\label{sec:proofs}

We are now in a position to construct the Lefschetz fibrations that were
promised in \S\ref{sec:results}.  It will be convenient to introduce the 
following notation.

Suppose $(W,\omega)$ is a strong filling of $(M,\xi)$ and $Y$ is a Liouville
vector field near $\p W$ such that $\iota_Y\omega|_M = e^f \lambda$ for some
contact form $\lambda$ on $M$ and smooth function $f : M \to \RR$.
Then for any constant $R > \max f$,
we can use Lemma~\ref{lemma:attachStein} to
attach the trivial symplectic cobordism 
$(\sS_f^{R},d(e^a\lambda))$, producing an enlarged filling
$$
(W^{R},\omega) := (W,\omega) \cup_{\p W} (\sS_f^{R},d(e^a\lambda)).
$$
This has $\p_a$ as a Liouville vector field near $\p W^{R}$, such that
$\iota_{\p_a}\omega|_{\p W^{R}} = e^{R}\lambda$.  
One can now attach a cylindrical end,
$$
(W^\infty,\omega) := (W^{R},\omega) \cup_{\p W^{R}}
([R,\infty) \times M, d(e^a\lambda)), 
$$
defining a noncompact symplectic cobordism which admits the compactification
$$
\overline{W}^\infty = W^{R} \cup_{\p W}
\left([R,\infty] \times M \right).
$$
We assign a smooth structure to $[R,\infty]$ so that $\overline{W}^\infty$
may be
considered a smooth manifold with boundary, though its symplectic structure
degenerates at $\p\overline{W}^\infty$.  It is sometimes useful however to
define a new symplectic structure on $W^\infty$ that does extend
to infinity.  Observe first that for any $\epsilon > 0$ with $R - \epsilon >
\max f$, $(W^\infty,\omega)$ contains the slightly extended cylindrical end
$([R - \epsilon,\infty) \times M, d(e^a\lambda))$.  Now choose
$\delta \in (0,\epsilon)$ and a diffeomorphism
$$
\varphi : [R -\epsilon,\infty] \to [e^{R-\epsilon},e^{R}]
$$
with the property that $\varphi(a) = e^a$ for $a \in [R-\epsilon,
R-\delta]$.
Then the symplectic form $\omega_{\varphi}$ on $W^\infty$ defined by
$$
\omega_\varphi =
\begin{cases}
d(\varphi\lambda) & \text{ on $[R -\epsilon,\infty) \times M$,}\\
\omega & \text{ everywhere else}
\end{cases}
$$
has a smooth extension to $\overline{W}^\infty$, such that the map
$$
[R - \epsilon,R] \times M \to [R - \epsilon,\infty] \times M :
(a,m) \mapsto (\varphi^{-1}(e^a),m)
$$
extends to a symplectomorphism $(W^{R},\omega) \to
(\overline{W}^\infty,\omega_\varphi)$.

We will consider almost complex structures $J$ on $W^\infty$ that are 
compatible with $\omega$, are generic in $W^\infty \setminus
\left( [R - \delta,\infty) \times M \right)$ and match some fixed
$J_+ \in \jJ_{\lambda}(M)$ over $[R - \delta,\infty) \times M$.
Observe that such a $J$ is also compatible with the modified symplectic form
$\omega_\varphi$ defined above, thus finite energy embedded
$J$--holomorphic curves in $W^\infty$ give rise to properly embedded
symplectic submanifolds of $(\overline{W}^\infty,\omega_\varphi) \cong
(W^{R},\omega)$.

\begin{lemma}
\label{lemma:spheres}
The almost complex structure $J$ above can be chosen so that every closed,
nonconstant $J$--holomorphic curve in $(W^\infty,J)$ is contained in 
the interior of~$W$.
\end{lemma}
\begin{proof}
\red{It suffices to arrange that $W^\infty \setminus W$ is foliated by
$J$--convex hypersurfaces.  Choose $r < R - \delta$,
let $h : [r,\infty) \times M \to \RR$ denote any smooth function 
satisfying
\begin{enumerate}
\item $\p_a h > 0$,
\item $h(a,m) = a$ for $a \ge R - \delta$,
\item $h(a,m) = a - r + f(m)$ for $a$ near~$r$,
\end{enumerate}
and define a diffeomorphism
$$
\psi : [r,\infty) \times M \to \sS_f^\infty : (a,m) \mapsto (h(a,m),m).
$$
This restricts to the identify on $[R-\delta,\infty) \times M$ and
satisfies $\psi^*(e^a\lambda) = e^h\lambda$, thus it defines a
symplectomorphism $([r,\infty) \times M, d(e^h\lambda)) \to 
(\sS_f^R, d(e^a\lambda))$.  Now for $a \in [r,\infty)$, denote by
$h_a  : M \to (0,\infty)$ the smooth $1$--parameter family of
functions such that $e^{h(a,\cdot)} = e^a h_a$, and define the family of
contact forms $\lambda_a := h_a\lambda$ with corresponding Reeb vector
fields~$X_a$.  Regarding $\lambda_a$ in the natural way as a $1$--form on
$\RR\times M$, we now have
$$
d(e^h\lambda) = e^a\ da \wedge \lambda_a + e^a \ d\lambda_a,
$$
and an almost complex structure $\hat{J}$ compatible with 
$d(e^h\lambda)$ can thus be constructed as follows.
Given $J_+ \in \jJ(\lambda)$, choose $\hat{J}$ on $[r,\infty) \times M$ 
so that it matches $J_+$ on $[R-\delta,\infty) \times M$, and
at $\{a\} \times M$ satisfies
$$
J \p_a = X_a \qquad \text{ and } \qquad J(\xi) = \xi,
$$
where $J|_\xi$ is compatible with $d\lambda$ (and therefore also with
$d\lambda_a$ for each $a$).  Now the 
level sets $\{a\} \times M$ are $\hat{J}$--convex, 
thus an almost complex structure
of the desired form on $\sS_f^\infty$ is given by $J := \psi_*\hat{J}$, and we
can extend the latter
to an $\omega$--compatible almost complex structure on~$W^\infty$ for which
the hypersurfaces $\psi(\{a\} \times M)$ for $a \ge r$ are $J$--convex.
Since $J$--convexity is an open condition
with respect to~$J$, it is also safe to make a small perturbation
on $W^R$ so that $J$ becomes generic outside of $[R-\delta,\infty) \times M$.
}
\end{proof}

\begin{proof}[Proof of Theorem~\ref{thm:Lefschetz}]
Assume $(M,\xi)$ is a contact manifold supported by a planar open book
$\pi : M \setminus B \to S^1$.  Then using the construction in 
\cite{Wendl:openbook}, there is a nondegenerate contact form
$\lambda$ with $\ker\lambda = \xi$ and $J_+ \in \jJ_{\lambda}(M)$ 
such that up to isotopy, the pages of $\pi$ are projections to $M$ of 
embedded $J_+$--holomorphic
curves in $\RR\times M$, with positive ends asymptotic to the orbits in~$B$.
This defines a positive finite energy foliation $\fF_+$ of $(M,\lambda,J_+)$, 
with every leaf stable.  Now if $(W,\omega)$ is a strong filling of $(M,\xi)$,
we define the enlarged fillings $W^R$ and $W^\infty$ with generic almost
complex structure $J$ as described above, and then
Theorem~\ref{thm:compactness} yields a moduli space $\overline{\mM}_0$ of 
$J$--holomorphic curves
that foliate $W^\infty$ outside a finite set of transverse nodes,
such that $\p\overline{\mM}_0$ is the space of leaves in $\fF_+$ up to
$\RR$--translation.  Since $\lambda$ is nondegenerate, $\overline{\mM}_0 \cong
\DD$, and the map
$$
\Pi : \overline{W}^\infty \setminus B \to \overline{\mM}_0
$$
defined as in \eqref{eqn:Pi} gives a symplectic Lefschetz fibration of
$(\overline{W}^\infty\setminus B,\omega_\varphi) \cong 
(W^{R} \setminus B,\omega)$ over the disk.
We can easily modify $\Pi$ so that it extends over $B$: first fatten $B$
to a tubular neighborhood $\nN(B) \subset M$, then extend $\Pi$ over this
neighborhood by contracting the disk.  We observe finally that if any
singular fiber contains a closed component, this must be a holomorphic
sphere $v : S^2 \to W^\infty$ with $i(v ; v) = -1$, thus an exceptional sphere,
and for an appropriate choice of $J$ it must be contained in $W$ due to 
Lemma~\ref{lemma:spheres}.
Therefore if $W$ is minimal, every component of a singular
fiber has nonempty boundary, implying that the vanishing cycle is homologically
nontrivial.
\end{proof}

\begin{proof}[Proof of Theorem~\ref{thm:T3}]
The argument is \red{mostly} the same as for Theorem~\ref{thm:Lefschetz}, but 
using a specific Morse-Bott finite energy foliation constructed as in
Example~\ref{ex:torsion} (see Remark~\ref{remark:foliationT3}).
In this case the space of leaves in $T^3$ is parametrized
by two disjoint circles, thus the moduli space $\overline{\mM}_0$ provided
by Theorem~\ref{thm:compactness} has two boundary components, and is therefore
an annulus.  The argument produces a Lefschetz fibration
$\Pi : \overline{W}^\infty \setminus Z \to [0,1]\times S^1$, which one can extend 
over $Z$ by fattening it to a neighborhood $\nN(Z)$ and then filling in 
using the homotopy between components of $\p\overline{\mM}_0$.

\red{It remains to show that all singular fibers consist of a union of
a cylinder with an exceptional sphere.  By Theorem~\ref{thm:compactness},
the only other option is a union of two transversely intersecting disks,
which would give a vanishing cycle parallel to the boundary of the fiber.
We can rule this out by looking at the monodromy maps of the fibrations
at $\{0\} \times S^1$ and $\{1\} \times S^1$: these are the two connected
components of the fibration in \eqref{eqn:piT3}.  Thus both monodromy
maps are trivial, but they must also be related to each other by a product
of positive Dehn twists, one for each nontrivial vanishing cycle.
Since the mapping class group of
the cylinder has only one generator, there is no product of positive Dehn
twists that gives the identity, thus there can be no nontrivial vanishing 
cycles.}
\end{proof}

\begin{proof}[Proof of Theorem~\ref{thm:stability}]
For a smooth $1$--parameter family of strong fillings $(W,\omega_t)$ of 
$(M,\xi)$ with $t \in [0,1]$ and a suitable Morse-Bott contact form $\lambda$,
one can find a smooth family of functions 
$f_t : M \to \RR$ such that for $R > \max \{ f_t(m)\ |\ t \in [0,1],
m \in M \}$, the trivial symplectic cobordism $(\sS_{f_t}^R, d(e^a\lambda))$
can be attached to $(W,\omega_t)$, producing an enlarged filling
$(W^R,\omega_t)$ whose symplectic form is fixed near the boundary.
Now attach the cylindrical end as usual and choose a generic smooth 
$1$--parameter family $J_t$ of $\omega_t$--compatible almost complex 
structures that are identical on the end.  If $(M,\xi)$ is planar or is
$(T^3,\xi_0)$, then the result now follows by applying the same
arguments as in the previous two proofs together with
Theorem~\ref{thm:compactnessHomotopy}.
\end{proof}

\begin{proof}[Proof of Theorem~\ref{thm:obstruction}] 
Suppose $(M,\xi)$ is a contact manifold with a positive foliation $\fF$ of
$(M_0,\lambda,J)$ containing an interior stable leaf $u \in \fF$ that is
asymptotically simple: then for any strong filling $(W,\omega)$, we can 
again fill $W^\infty$ with $J$--holomorphic 
curves using Theorem~\ref{thm:compactness}, and we already have a
contradiction if $M_0 \subsetneq M$.  On the other hand if $M_0 = M$,
we can find a point $p$ that lies in some ``different'' leaf $u' \in \fF$,
and then consider for large $n$ the sequence
$u_n \in \mM_0$, where $u_n$ is the unique curve
passing through $(n,p) \in [R,\infty) \times M \subset W^\infty$.  As
$n \to \infty$, a subsequence must converge to $u'$, implying that $u$ 
and $u'$ are diffeomorphic and have ends in
the same Morse-Bott manifolds, which is a contradiction.
\end{proof}

\begin{remark}
\label{remark:Weinstein}
The \emph{Weinstein conjecture} for a contact manifold $(M,\xi)$ asserts that
for any contact form $\lambda$ with $\ker\lambda = \xi$, $X_\lambda$ has
a periodic orbit.  The idea of using punctured holomorphic curves to prove
this is originally due to Hofer \cite{Hofer:weinstein}, and works so far
under a variety of assumptions on $(M,\xi)$ (see also 
\cite{ACH}).  The conjecture for general contact $3$--manifolds
was proved recently by Taubes \cite{Taubes:weinstein}, using Seiberg-Witten
theory, but a general proof using only holomorphic curves is still lacking.

A minor modification of Theorem~\ref{thm:compactness} yields a new proof of
the Weinstein conjecture for any setting in which one can construct a
positive foliation containing an interior stable leaf that is asymptotically
simple,
for instance on the standard $3$--torus, or any contact manifold with
positive Giroux torsion.  The argument is a generalization
of the one used by Abbas-Cieliebak-Hofer \cite{ACH} for planar
contact structures: we replace the
symplectic filling $W$ by a cylindrical symplectic cobordism $\widehat{W}$, 
having $(M, c\lambda)$ for some large constant $c > 0$ at the positive end and
$(M, f\lambda)$ for any smooth positive function $f : M \to \RR$ with
$f < c$ at the negative end.  Then the same compactness argument works
for any sequence of curves $u_n : \dot{\Sigma} \to \widehat{W}$ that is
bounded away from the negative end.  Just as in \cite{ACH},
one can therefore produce a sequence
$u_n$ that runs to $-\infty$ in the negative end and breaks along a
periodic orbit in $(M,f\lambda)$, proving the existence of such an 
orbit.\footnote{The compactness argument in \cite{ACH} contains a minor
gap, as it ignores the possibility of nodal degenerations.  Our argument
fills the gap by showing that only embedded \red{index~$0$} curves can appear 
in such degenerations, thus \red{they are confined to a subset of 
codimension~$2$ and can be avoided by following a generic path to~$-\infty$.}}
\end{remark}

\section{\red{Fillings of $T^3$}}
\label{sec:T3}

We now proceed to the proofs of Theorems~\ref{thm:SteinT3}
and~\ref{thm:sympT3} on fillings of~$T^3$.  \red{The key fact is that
if a strong filling of $(T^3,\xi_0)$ is minimal, then the Lefschetz
fibration given by Theorem~\ref{thm:T3} is an honest symplectic fibration,
i.e.~it has no singular fibers.}  In fact, it is easy to construct two such
fibrations, whose fibers intersect each other exactly once transversely;
the situation is thus analogous to that of Gromov's characterization
of split symplectic forms on $S^2 \times S^2$ (\cite{Gromov}, also
subsequent related work by McDuff \cite{McDuff:rationalRuled}).
We can construct a simple model Stein manifold, which is
symplectomorphic to $T^*T^2$ and carries an
explicit decomposition by two fibrations for which the complex and symplectic
structures both split.  Matching this decomposition with the fibrations
constructed for a general filling via Theorem~\ref{thm:compactness}
gives a \red{symplectic deformation equivalence, which
in the exact case} yields a symplectomorphism via the
Moser isotopy trick.  

There is one subtle point here that doesn't
arise in the closed case: since we intend to carry out the Moser isotopy
on a noncompact manifold, it's important that our diffeomorphism be
sufficiently well behaved near infinity, and this will not generally
be the case without some effort.  \red{To see why not, observe that for
any strong filling $(W,\omega)$ of $(T^3,\xi_0)$, the asymptotics of the
$J$--holomorphic curves in $W^\infty$ given by Theorem~\ref{thm:compactness} 
encode a homotopy invariant of the foliation.  Indeed, suppose
$\{ \gamma_0^\eta \}_{\eta \in S^1}$ and $\{ \gamma_1^\eta \}_{\eta \in S^1}$
are the two Morse-Bott families of Reeb orbits that serve as the asymptotic
limits of the curves in the moduli space~$\mM$.  Then we can
choose a diffeomorphism 
$$
\RR \times S^1 \to \mM : (\rho,\eta) \mapsto u_{(\rho,\eta)}
$$
such that $u_{(\rho,\eta)}$ has asymptotic orbits $\gamma_0^\eta$ and
$\gamma_1^{f(\rho,\eta)}$ for some continuous function 
$$
f : \RR \times S^1 \to S^1,
$$
which has the form $f(\rho,\eta) = \eta$ for $|\rho|$ large due to the
fixed structure of $\mM$ in the cylindrical end.
The map $\rho \mapsto f(\rho,0)$ thus defines a loop in $S^1$ whose
homotopy class in $\pi_1(S^1) = \ZZ$ can be shown
(using Theorem~\ref{thm:compactnessHomotopy}) to be an invariant determined 
by $(W^\infty,\omega)$ 
and~$J$ up to compactly supported deformations.  Now if
$(W_1,\omega_1)$ and $(W_2,\omega_2)$ are two strong fillings that we wish
to prove are symplectomorphic, we'd like to do so by choosing a 
diffeomorphism that both respects the structure of the holomorphic foliations
and is ``compactly supported'' in the sense of respecting the natural
identifications of $W_1^\infty$ and $W_2^\infty$ with 
$[R,\infty) \times T^3$ near infinity.
It is easy enough to modify the foliations slightly so that an appropriate
diffeomorphism can be constructed near infinity, but this will not be
globally extendable unless the above construction gives the same class 
in $\pi_1(S^1)$ for both foliations.  

The upshot is that it is not enough to take only $T^*T^2$ with its
standard complex and symplectic structure as a model filling---rather, we
will need a wider variety of models that come with holomorphic foliations
attaining all possible values in $\pi_1(S^1)$.  We'll construct such
models in \S\ref{subsec:SteinModels}} by performing Luttinger surgery
along the zero section in $T^*T^2$.  Note that unlike the situation in a
closed manifold, the manifolds obtained by surgery are all symplectomorphic,
but the point is that their complex structures (and the resulting
holomorphic foliations) behave differently at infinity.
With these models in place, we'll carry out the Moser deformation argument
in \S\ref{subsec:Moser} to prove Theorem~\ref{thm:SteinT3}.  Finally,
\S\ref{subsec:sympT3} will use the stability of our fibrations under
homotopies (Theorem~\ref{thm:compactnessHomotopy}) to prove
Theorem~\ref{thm:sympT3}.

\subsection{Model fillings and fibrations}
\label{subsec:SteinModels}

As usual, we identify $T^*T^2$ with $T^2 \times \RR^2$ and use coordinates
$(q_1,q_2,p_1,p_2)$, so that the standard symplectic structure is
$\omega_0 = d\lambda_0$, where $\lambda_0 = p_1 \ dq_1 + p_2 \ dq_2$.
Each pair of coordinates $(p_j,q_j)$ for $j = 1,2$ defines a cylinder
$Z_j = \RR \times S^1$ so that we have a natural diffeomorphism
$$
T^2 \times \RR^2 = Z_1 \times Z_2.
$$
We define on each $Z_j$ the standard complex structure
$i \p_{p_j} = \p_{q_j}$ and symplectic structure $\omega_0 =
d p_j \wedge d q_j$, so that $\omega_0$ on $Z_1 \times Z_2$ is the direct
sum $\omega_0 \oplus \omega_0$, and
we can similarly define a compatible complex structure $i$ on 
$T^2 \times \RR^2$ as $i \oplus i$.
This makes $(T^2 \times \RR^2,\omega_0,i)$ into a Stein manifold, with
plurisubharmonic function $f : T^2 \times \RR^2 \to [0,\infty) :
(q,p) \mapsto \frac{1}{2}|p|^2$ such that $-df \circ i = \lambda_0$,
and the latter induces the Liouville vector field
$$
\nabla f = p_1 \p_{p_1} + p_2 \p_{p_2},
$$
whose flow is given by $\varphi^t_{\nabla f}(q,p) = (q,e^t p)$.
The restriction of $\lambda_0$ to $\p(T^2 \times \DD) = T^3$
gives the standard contact form, which we'll denote in the following
by~$\alpha_0$.  We will use the 
coordinates $(q,p)$ on $T^3$ with the assumption that $|p| = 1$, and
sometimes also write $(p_1,p_2) = (\cos 2\pi\theta,\sin 2\pi\theta)$
with $\theta \in S^1$.

We can use the flow of $\nabla f$ to embed the symplectization of $T^3$
into $(T^2 \times \RR^2,\omega_0)$: explicitly,
$$
\Phi : (\RR\times T^3, d(e^a\alpha_0)) \hookrightarrow 
(T^2 \times \RR^2,\omega_0) : (a,(q,p)) \mapsto (q,e^a p)
$$
satisfies $\Phi^*\lambda_0 = e^a \alpha_0$.
Using this to identify $(0,\infty) \times T^3$ with the complement of
$T^2 \times \DD$, we can now choose a new almost complex structure
$J_0$ with $J_0 \p_{p_j} = g(|p|) \p_{q_j}$
for some function $g$, so that $J_0 = i$ near the zero section and becomes
$\RR$--invariant on the end, in other words
$J_0|_{[0,\infty) \times T^3} \in \jJ_{\alpha_0}(T^3)$.  This choice of
$J_0$ has precisely the form on $[0,\infty) \times T^3$ that
was used in Example~\ref{ex:torsion} (via Remark~\ref{remark:foliationT3}).
In terms of the splitting $T^2 \times \RR^2 = Z_1 \times Z_2$,
the cylinders $Z_1 \times \{*\}$ and $\{*\} \times Z_2$ are now finite
energy $J_0$--holomorphic curves, and those which lie entirely in
$[0,\infty) \times T^3$ reproduce the 
foliations constructed in Example~\ref{ex:torsion}.  In particular,
each cylinder $Z_1 \times \{*\}$ is asymptotic to a pair of Reeb orbits
in the Morse-Bott tori $\{ \theta = 0, 1/2 \}$ with the same value of
the coordinate $q_2 \in S^1$ at both ends, and a corresponding statement is 
true for $\{*\} \times Z_2$ with the Morse-Bott tori $\{ \theta = 1/4, 3/4 \}$.

\red{We shall now construct more holomorphically foliated model fillings
using surgery along
the zero section in $T^2 \times \RR^2$.}  The following is a special case
of the surgery along a Lagrangian $2$--torus in a symplectic $4$--manifold
introduced by Luttinger in \cite{Luttinger}; our formulation is borrowed from
\cite{ADK:Luttinger}.

For $r > 0$, let $K_r = T^2 \times [-r,r] \times [-r,r]$.  Choose constants 
$\sigma := (c,k_1,k_2) \in (0,\infty) \times \ZZ^2$
and a smooth cutoff function 
$\beta : \RR \to [0,1]$ such that
\begin{itemize}
\item
$\beta = 0$ on a neighborhood of $(-\infty,-1]$,
\item
$\beta = 1$ on a neighborhood of $[1,\infty)$,
\item
$\int_{-1}^1 t \beta'(t)\ dt = 0$.
\end{itemize}
Define also the function
$\chi : \RR \to \RR$ to equal $0$ on $(-\infty,0)$ and $1$ on $[0,\infty)$.
Then there is a symplectomorphism
$\psi_{\sigma} : (K_{2c} \setminus K_{c},\omega_0) \to 
(K_{2c} \setminus K_{c},\omega_0)$ given by
$$
\psi_{\sigma}(q_1,q_2,p_1,p_2) = \left(q_1 + k_1 \chi(p_2) 
\beta\left(\frac{p_1}{c} \right),
q_2 + k_2 \chi(p_1) \beta\left(\frac{p_2}{c}\right), p_1,p_2\right).
$$
We construct a new symplectic manifold 
$(W_{\sigma},\omega_{\sigma})$ by deleting
$K_{c}$ from $T^2 \times \RR^2$ and gluing in
$K_{2c}$ via $\psi_{\sigma}$:
$$
(W_{\sigma},\omega_{\sigma}) = 
((T^2 \times \RR^2) \setminus K_{c},\omega_0)
\cup_{\psi_{\sigma}} (K_{2c},\omega_0).
$$
In the following, we shall regard both 
$((T^2 \times \RR^2) \setminus K_{c},\omega_0)$
and $(K_{2c},\omega_0)$ as symplectic subdomains of 
$(W_{\sigma},\omega_{\sigma})$, \red{and fix local coordinates as follows.
Let $(q_1,q_2,p_1,p_2)$ denote the usual coordinates on
$(T^2 \times \RR^2) \setminus K_c$, now viewed as a subset of $W_\sigma$,
and on the glued in copy of $K_{2c} \subset W_\sigma$, denote the natural
coordinates by $(Q_1,Q_2,P_1,P_2)$.  Thus on the region of overlap,
$(q,p) = \psi_\sigma(Q,P)$ and
$$
\omega_\sigma = d p_1 \wedge d q_1 + d p_2 \wedge d q_2 =
d P_1 \wedge d Q_1 + d P_2 \wedge d Q_2.
$$
Observe that the $(Q,P)$--coordinates can be extended globally so that they
define a symplectomorphism
$(Q,P) : (W_{\sigma},\omega_\sigma) \to (T^2 \times \RR^2,\omega_0)$.}

If $2c = e^R$, then the part of $(W_{\sigma},\omega_{\sigma})$ 
identified
with $((T^2 \times \RR^2) \setminus K_{c},\omega_0)$ naturally contains a
symplectization end of the form $([R,\infty) \times T^3, d(e^a\alpha_0))$.

\begin{lemma}
\label{lemma:globalPrimitive}
$W_\sigma$ admits a $1$--form $\lambda_\sigma$ such that
$d\lambda_\sigma = \omega_\sigma$ and 
$\lambda_\sigma|_{[R,\infty) \times T^3} = e^a\alpha_0$.
\end{lemma}
\begin{proof}
The $1$--form $e^a\alpha_0$ is the restriction to $[R,\infty) \times T^3$
of $\lambda_0 := p_1\ dq_1 + p_2\ dq_2$, which is a well defined primitive
of $\omega_0 = \omega_\sigma$ on $(T^2 \times \RR^2) \setminus K_c$.
Define $f(s) = \frac{1}{c} \int_{-c}^s t \beta'(t/c)\ dt$, a smooth function
with support in
$(-c,c)$ due to our assumptions on~$\beta$.  Then there is a smooth
function $\Phi : (T^2 \times \RR^2) \setminus K_c \to \RR$ defined by
$$
\Phi(q_1,q_2,p_1,p_2) =
\begin{cases}
k_2 f(p_2) & \text{ if $p_1 \ge c$,}\\
k_1 f(p_1) & \text{ if $p_2 \ge c$,}\\
0      & \text{ otherwise,}
\end{cases}
$$
and a brief computation shows that on $(T^2 \times \RR^2) \setminus K_c$,
$\lambda_0 = P_1\ dQ_1 + P_2\ dQ_2 + d\Phi$.  Now choosing a smooth
function $\widehat{\Phi} : W_\sigma \to \RR$ that matches $\Phi$
on $[R,\infty) \times T^3$ and vanishes in $K_c$, a suitable primitive
is given by
$$
\lambda_\sigma = P_1\ dQ_1 + P_2\ dQ_2 + d\widehat{\Phi}.
$$
\end{proof}

We wish to define an $\omega_\sigma$--compatible almost complex structure
$J_\sigma$ on $W_\sigma$ that matches $J_0$ on the end
$[R,\infty) \times T^3$, i.e.~for $|p| \ge e^R$, $J_\sigma$ satisfies
$- J_\sigma \p_{q_j} = G(|p|) \p_{p_j}$ for some positive smooth function~$G$.  
Switching to $(Q,P)$--coordinates in $K_{2c}$,
$J_\sigma$ is now determined in 
$K_{2c} \cap \left([R,\infty) \times T^3\right)$ by the conditions
\begin{equation*}
\begin{split}
-J_\sigma \p_{Q_1} &= \p_{P_1} - G(|P|) \frac{k_1}{c} 
\chi(P_2) \beta'(P_1 / c) \ \p_{Q_1},\\
-J_\sigma \p_{Q_2} &= \p_{P_2} - G(|P|) \frac{k_2}{c}
\chi(P_1) \beta'(P_2 / c) \ \p_{Q_2}.
\end{split}
\end{equation*}
Thus if we replace $\chi$ in this expression by the cutoff function 
$t \mapsto \beta(t/c)$, which equals $\chi$ outside of
$[-c,c]$, we obtain the desired extension of $J_\sigma$ over $K_{2c}$.
The following lemma is immediate.

\begin{lemma}
\label{lemma:skewedFoliation}
For each constant $(\rho,\eta) \in \RR \times S^1$, the surfaces 
$Z_1^{(\rho,\eta)} := \{ (P_2,Q_2) = (\rho,\eta) \}$ and 
$Z_2^{(\rho,\eta)} := \{ (P_1,Q_1) = (\rho,\eta) \}$ in $W_{\sigma}$
are images of embedded finite energy $J_\sigma$--holomorphic cylinders.
Moreover,
\begin{enumerate}
\item
Each point in $W_\sigma$ is the unique intersection point of a unique
pair $Z_1^{(\rho,\eta)}$ and $Z_2^{(\rho',\eta')}$,
whose tangent spaces at that point are symplectic complements.
\item
For $|\rho| \ge c$, the cylinders $Z_1^{(\rho,\eta)}$ and
$Z_2^{(\rho,\eta)}$ are identical to $Z_1 \times \{(\rho,\eta)\}$
and $\{(\rho,\eta)\} \times Z_2$ respectively in $T^2 \times \RR^2 =
Z_1 \times Z_2$.  This collection therefore contains all of the curves
in $[R,\infty) \times T^3$ constructed via 
Example~\ref{ex:torsion} and Remark~\ref{remark:foliationT3}.
\end{enumerate}
\end{lemma}

The essential difference between $(W_\sigma,\omega_\sigma)$ and
$(T^2 \times \RR^2,\omega_0)$ is that they each come with holomorphic
foliations that behave differently at infinity: the cylinder
$Z_1^{(\rho,\eta)}$ for instance has one end asymptotic to the Reeb orbit at
$\{ \theta=1/2, q_2 = \eta \}$, while its other end approaches the orbit at
$\{ \theta=0, q_2 = \eta + k_2\beta(\rho/c) \}$.  Thus the data
$\sigma = (c,k_1,k_2)$ determine offsets within the respective 
families of Morse-Bott orbits at one end of each cylinder.

\subsection{Classification up to symplectomorphism}
\label{subsec:Moser}

Assume $(W,\omega)$ is a \red{minimal strong} filling of $(T^3,\xi_0)$.
Adopting the notation from \S\ref{sec:proofs}, $(W^{R},\omega)$ is the enlarged
filling obtained by attaching a trivial symplectic cobordism
such that the induced contact form at $\p W^{R}$ is $e^R \alpha_0$,
and we can further attach a cylindrical end 
$([R,\infty) \times T^3, d(e^a\alpha_0))$ to
construct $(W^\infty,\omega)$.  \red{If $(W,\omega)$ is an exact filling with
primitive $\lambda$, then we can also assume $\lambda$ extends over $W^\infty$
so that $\lambda|_{[R,\infty) \times T^3} = e^a\alpha_0$.}
Choosing an almost complex structure
$J$ that is generic in $W^{R}$ and has the standard form
$J_0 \in \jJ_{\alpha_0}(T^3)$ on $[R,\infty) \times T^3$, 
we start from a finite energy
foliation constructed as in Example~\ref{ex:torsion} (via
Remark~\ref{remark:foliationT3}), consisting of cylinders with ends
asymptotic to orbits in the two Morse-Bott tori $Z = \{ \theta \in
\{0, 1/2 \} \}$, then use Theorem~\ref{thm:compactness} 
to produce a moduli space $\mM_1$ of $J$--holomorphic cylinders foliating
$W^\infty$.  \red{Since $(W,\omega)$ is minimal, this produces a} smooth
fibration $\Pi_1 : W^\infty \to \mM_1$, where both the fiber and the
base are diffeomorphic to $\RR\times S^1$.

We can now repeat the same trick starting from a different foliation
of $T^3$: let $Z' = \{ \theta \in \{1/4,3/4\} \}$, a pair of
Morse-Bott tori with Reeb orbits pointing in the direction orthogonal
to those on~$Z$.  Then by a minor modification of the construction in
Example~\ref{ex:torsion}, the fibration
\begin{equation*}
\begin{split}
T^3 \setminus Z' &\to \{0,1\} \times S^1 \\
({q_1},{q_2},\theta) &\mapsto
\begin{cases}
(0,q_1) & \text{ if $\theta \in (-1/4,1/4)$,}\\
(1,q_1) & \text{ if $\theta \in (1/4,3/4)$}
\end{cases}
\end{split}
\end{equation*}
can also be presented as the projection to $T^3$ of a positive finite energy
foliation on $\RR \times T^3$, with the same contact form and almost
complex structure as before.  This yields a second moduli space 
$\mM_2$ of $J$--holomorphic cylinders foliating $W^\infty$, and a
corresponding fibration $\Pi_2 : W^\infty \to \mM_2 \cong
\RR \times S^1$.

\begin{lemma}
\label{lemma:oneIntersection}
Any $u_1 \in \mM_1$ and $u_2 \in \mM_2$ intersect each other exactly once,
with intersection index~$+1$.
\end{lemma}
\begin{proof}
One can verify this explicitly from the foliations on $[R,\infty) \times T^3$
whenever both curves are near the boundaries
of their respective moduli spaces, and since they have
no asymptotic orbits in common, this implies $i(u_1;u_2) = 1$.  The latter is
a homotopy invariant condition, and the fact that the two curves have separate
orbits guarantees that there is never any asymptotic contribution, hence
there is always a unique intersection point $u_1(z_1) = u_2(z_2)$, 
contributing $+1$ to the intersection count.
\end{proof}

It follows that the map
$$
\Pi_1 \times \Pi_2 : W^\infty \to \mM_1 \times \mM_2
$$
is a diffeomorphism.  Our goal is to use this to identify $W^\infty$ with 
one of the model \red{fillings} constructed in \S\ref{subsec:SteinModels}.

For $\theta \in \{0,1/4,1/2,3/4\}$, denote by $\pP_{\theta}$ the
$1$--dimensional manifold of Morse-Bott orbits foliating the $2$--torus
whose $\theta$--coordinate has the given value:
each of these can be naturally identified with
$S^1$ using either the $q_1$ or $q_2$--coordinate.  Then
as explained in the appendix, there exist real line bundles
$$
E^\theta \to \pP_\theta,
$$
where the fibers $E^\theta_x$ are $1$--dimensional eigenspaces of the
asymptotic operators at $x \in \pP_\theta$, and the asymptotic formula
\eqref{eqn:asympFormula} defines ``asymptotic evaluation maps''
\begin{align*}
\mM_1 &\xrightarrow{\ev_0} E^0 &
\mM_1 &\xrightarrow{\ev_{1/2}} E^{1/2} \\
\mM_2 &\xrightarrow{\ev_{1/4}} E^{1/4} &
\mM_2 &\xrightarrow{\ev_{3/4}} E^{3/4}.
\end{align*}

For any $\sigma = (c,k_1,k_2) \in (0,\infty) \times \ZZ^2$,
let $\mM_1^\sigma$ and $\mM_2^\sigma$ denote the moduli spaces of
$J_\sigma$--holomorphic cylinders $Z_1^{(\rho,\eta)}$ and $Z_2^{(\rho,\eta)}$ 
respectively in $(W_\sigma,\omega_\sigma)$, constructed in the
previous section: as a special case,
$\mM_1^0$ and $\mM_2^0$ will denote the spaces of $J_0$--holomorphic
cylinders $Z_1 \times \{*\}$, $\{*\} \times Z_2$ in 
$(T^2\times \RR^2,\omega_0)$.  \red{These last two moduli spaces} 
are each canonically identified
with $\RR \times S^1$, and they also come with asymptotic evaluation
maps $\ev_\theta^0$, defined as above.  These are manifestly diffeomorphisms 
and have the property that the resulting maps
\begin{equation}
\label{eqn:evComposition}
\begin{split}
(\ev_\theta^0)^{-1} \circ \ev_\theta : \mM_1 \to \mM_1^0 &
\text{ for $\theta = 0,1/2$,}\\
(\ev_\theta^0)^{-1} \circ \ev_\theta : \mM_2 \to \mM_2^0 &
\text{ for $\theta = 1/4,3/4$}
\end{split}
\end{equation}
are proper: indeed, for any $u \in \mM_j$ outside of some compact subset,
they define the natural identification between curves in $\mM_j$ and
$\mM_j^0$ that are contained in the cylindrical end.

\begin{lemma}
The maps defined in \eqref{eqn:evComposition} are diffeomorphisms.
\end{lemma}
\begin{proof}
They are local diffeomorphisms due to Lemma~\ref{lemma:evalInfty}.  
The claim thus reduces to the fact that any local diffeomorphism with 
compact support on a cylinder $\RR\times S^1$ is a global diffeomorphism.
\end{proof}

By the lemma, we can compose \eqref{eqn:evComposition} with the canonical
identifications $\mM_j^0 = \RR\times S^1$ and define diffeomorphisms
\begin{equation*}
\begin{split}
\varphi_\theta : \mM_1 \to \RR\times S^1 &
\text{ for $\theta = 0,1/2$,}\\
\varphi_\theta : \mM_2 \to \RR\times S^1 &
\text{ for $\theta = 1/4,3/4$,}
\end{split}
\end{equation*}
so that the resulting compositions $\varphi_0 \circ \varphi_{1/2}^{-1}$
and $\varphi_{1/4} \circ \varphi_{3/4}^{-1}$ are diffeomorphisms of
$\RR\times S^1$ with compact support.  Choose $c > 0$ sufficiently large
so that both of these are supported in $[-c,c] \times S^1$ and (making $R$
larger if necessary) $2c = e^R$.  Now, recalling the cutoff function 
$\beta$ from \S\ref{subsec:SteinModels}, set $\sigma = (c,k_1,k_2)$
where $k_1,k_2$ are the unique integers such that
there is an isotopy $\{\psi_t^1 \in \Diff(\RR\times S^1) \}_{t
\in [0,1]}$ supported in $[-c,c] \times S^1$, with $\psi_0^1 =
\varphi_0 \circ \varphi_{1/2}^{-1}$ and
$$
\psi_1^1(\rho,\eta) = (\rho, \eta + k_2 \beta(\rho / c)),
$$
and similarly there is an isotopy $\psi_t^2$ from $\varphi_{1/4} \circ
\varphi_{3/4}^{-1}$ to
$$
\psi_1^2(\rho,\eta) = (\rho, \eta + k_1 \beta(\rho / c)).
$$
From now on, we will use the diffeomorphisms $\varphi_{1/2}$ and
$\varphi_{3/4}$ to parametrize $\mM_1$ and $\mM_2$ respectively, denoting
$$
u_1^{(\rho,\eta)} := \varphi_{1/2}^{-1}(\rho,\eta),
\qquad
u_2^{(\rho,\eta)} := \varphi_{3/4}^{-1}(\rho,\eta).
$$
The point of this convention is that $u_1^{(\rho,\eta)} \in \mM_1$ now
approaches the Morse-Bott family $\{\theta = 1/2\}$ at the same orbit and
along the same asymptotic eigenfunction as $Z_1^{(\rho,\eta)} \in
\mM_1^\sigma$, and a corresponding statement holds
for $\mM_2$ and $\mM_2^\sigma$.

\begin{lemma}
\label{lemma:workhorse}
There exist constants $R_2 > R_1 > R$, an almost complex structure
$\hat{J}$ on $W^\infty$ tamed by $\omega$, and moduli spaces
$\widehat{\mM}_1$ and $\widehat{\mM}_2$ of embedded finite energy
$\hat{J}$--holomorphic cylinders foliating $W^\infty$, which have the
following properties.  For $j \in \{1,2\}$, 
$\widehat{\mM}_j$ can be parametrized by a cylinder
$$
\RR\times S^1 \ni (\rho,\eta) \mapsto \hat{u}_j^{(\rho,\eta)} \in
\widehat{\mM}_j
$$
such that
\begin{enumerate}
\item
In the region $W^R \cup ([R,R_1] \times T^3)$, $\hat{J} \equiv J$ and
$\hat{u}_j^{(\rho,\eta)}$ is identical to $u_j^{(\rho,\eta)} \in \mM_j$.
\item
In $[R_2,\infty) \times T^3$, $\hat{J} \equiv J_\sigma$ and
$\hat{u}_j^{(\rho,\eta)}$ is identical to
$Z_j^{(\rho,\eta)} \in \mM_j^\sigma$, where we use the natural identification
of the ends of $W^\infty$ and $W_\sigma$.
\item
Lemma~\ref{lemma:oneIntersection} holds also for the spaces
$\widehat{\mM}_1$ and $\widehat{\mM}_2$.
\end{enumerate}
\end{lemma}
\begin{proof}
The curves $u_j^{(\rho,\eta)}$ already have the desired properties when
$|\rho| \ge c$, so changes are needed only on compact subsets of
$\mM_j$, and only near the ends of these curves.  
The idea is simply to modify the foliation defined by
$\{ u_j^{(\rho,\eta)} \}_{(\rho,\eta) \in [-c,c]\times S^1}$ 
outside of a large compact subset to a new foliation of the
same region such that the change to the tangent spaces is uniformly small.
One can then make the new foliation $\hat{J}$--holomorphic for some
$\hat{J}$ that is uniformly close to $J$ and therefore also tamed by
$\omega$.  Lemma~\ref{lemma:oneIntersection} is trivial to verify
for the modified foliations, because adjustments to $\mM_1$ happen only
in a region where $\mM_2$ is unchanged, and vice versa.
We proceed in two steps.

Choose $R_1 > 0$ sufficiently large so that for $|\rho| \le c$, the tangent
spaces of the curves $u_j^{(\rho,\eta)}$ in $[R_1,\infty) \times T^3$ 
are uniformly close to the
tangent spaces of the asymptotic orbit cylinders.  Then choosing
$R'$ much larger than $R_1$, a sufficiently gradual adjustment
of the remainder
term in the asymptotic formula \eqref{eqn:asympFormula} produces a new
surface $\hat{u}_j^{(\rho,\eta)}$ in $[R_1,R'] \times T^3$ that looks like 
$u_j^{(\rho,\eta)}$ near $\{ R_1 \} \times T^3$ and
$Z_j^{(\rho',\eta')} \in \mM_j^0$ near $\{ R' \} \times T^3$, where
$(\rho',\eta')$ is related to $(\rho,\eta)$ via the diffeomorphism
$\varphi_0 \circ \varphi_{1/2}^{-1}$ or
$\varphi_{1/4} \circ \varphi_{3/4}^{-1}$.

It remains to adjust the parameters $(\rho',\eta')$ so that
in $[R_2,\infty) \times T^3$ for some $R_2 > R'$, $\hat{u}_j^{(\rho,\eta)}$
matches $Z_j^{(\rho,\eta)} \in \mM_j^\sigma$.
For this we use the isotopies $\psi^j_t$, defining the surface
$\hat{u}_j^{(\rho,\eta)}$ so that its intersection with
$\{s\} \times T^3$ for $s \in [R',R_2]$ matches
$Z_j^{\psi^j_{f(t)}(\rho,\eta)} \in \mM_j^0$ for some function 
$f : [R',R_2] \to [0,1]$ with sufficiently small derivative.
(Of course, $R_2$ must be large).
\end{proof}

We can now carry out the \red{deformation} argument.

\begin{prop}
\label{prop:Moser}
There exists a \red{diffeomorphism $\psi : W_\sigma \to W^\infty$
which restricts to the identity on $[R_2,\infty) \times T^3$, such that
the $2$--forms 
$$
\omega_{(t)} := t \psi^*\omega + (1-t) \omega_\sigma
$$
are symplectic for all $t \in [0,1]$.}
\end{prop}
\begin{proof}
Applying Lemma~\ref{lemma:oneIntersection} to the spaces
$\widehat{\mM}_1$ and $\widehat{\mM}_2$ and using the given identifications
of both with $\RR\times S^1$, we have a diffeomorphism
$$
\widehat{\Pi}_1 \times \widehat{\Pi}_2 : W^\infty \to 
\widehat{\mM}_1 \times \widehat{\mM}_2 = (\RR\times S^1) \times 
(\RR\times S^1),
$$
and there is a similar diffeomorphism
$$
\Pi_1^\sigma \times \Pi_2^\sigma : W_\sigma \to \mM_1^\sigma \times 
\mM_2^\sigma = (\RR\times S^1) \times (\RR\times S^1).
$$
Composing the second with the inverse of the first yields a diffeomorphism
$\psi : W_\sigma \to W^\infty$
which equals the identity in $[R_2,\infty) \times T^3$.
We claim that $\omega_{(t)} = t \psi^*\omega + (1-t) \omega_\sigma$ is 
nondegenerate, and
thus symplectic for every $t \in [0,1]$.  Indeed, the almost complex
structure $\psi^*\hat{J}$ tames $\omega_{(1)} = \psi^*\omega$, and it
also tames $\omega_{(0)} = \omega_\sigma$ since every tangent space now
splits into a sum of $\omega_\sigma$--symplectic complements that are
also $\psi^*\hat{J}$--invariant.  Thus $\psi^*\hat{J}$ is also
tamed by $\omega_{(t)}$ for every $t \in [0,1]$, proving the claim.
\end{proof}

\red{
\begin{prop}
\label{prop:Moser2}
If $(W,\omega)$ is an exact filling, then one can arrange the diffeomorphism 
of Prop.~\ref{prop:Moser} to be a symplectomorphism
$(W_\sigma,\omega_\sigma) \to (W^\infty,\omega)$.
\end{prop}
\begin{proof}
Let $\psi : W_\sigma \to W^\infty$ be the diffeomorphism constructed in
Prop.~\ref{prop:Moser}.
By Lemma~\ref{lemma:globalPrimitive}, there is a $1$--form $\lambda_\sigma$
on $W_\sigma$ that satisfies $d\lambda_\sigma = \omega_\sigma$ globally
and matches $\lambda = \psi^*\lambda = e^a\alpha_0$ 
on $[R_2,\infty) \times T^3$.
Now $\omega_{(t)} = d\lambda_{(t)}$, where
$\lambda_{(t)} = t \psi^*\lambda + (1-t) \lambda_\sigma$.
Define a time-dependent vector field $V_t$ on $W_\sigma$ by
$$
\omega_{(t)}(V_t,\cdot) = \lambda_\sigma - \psi^*\lambda.
$$
Since $\lambda_\sigma - 
\psi^*\lambda$ vanishes in $[R_2,\infty) \times T^3$, the flow
$\varphi^t_V$ of $V_t$ has compact support and
is well defined for all~$t$: the map
$$
\psi \circ \varphi^1_V : W_\sigma \to W^\infty
$$
then gives the desired symplectomorphism $(W_\sigma,\omega_\sigma) \to
(W^\infty,\omega)$.
\end{proof}
}

\begin{proof}[Proof of Theorem~\ref{thm:SteinT3}]
\red{By Prop.~\ref{prop:Moser}, $(W,\omega)$ is symplectically deformation
equivalent to an exact filling, so let us assume from now on that it is
exact.  Then by} Prop.~\ref{prop:Moser2}, there is
a symplectomorphism $\psi : (W^\infty,\omega) \to (W_\sigma,\omega_\sigma)$ 
which equals the identity in $[R,\infty) \times T^3$ for sufficiently
large $R$, and we shall now use it to construct a symplectomorphism of 
$(W,\omega)$ to
a star shaped domain in $T^*T^2$.  \red{Choose a global primitive $\lambda$
of $\omega$ which matches $e^a\alpha_0$ on $[R,\infty) \times T^3$ and
denote by $Y$ and $Y_\sigma$ the Liouville vector fields corresponding
to $\lambda$ and $\lambda_\sigma$ respectively, so}
$$
\omega(Y,\cdot) = \lambda,
\qquad
\omega_\sigma(Y_\sigma,\cdot) = \lambda_\sigma.
$$
Both of these match $\p_a$ on $[R,\infty) \times T^3$.
There is also another Liouville vector field $Y_0$ on $W_\sigma$ defined
by $\omega_\sigma(Y_0,\cdot) = P_1\ dQ_1 + P_2\ dQ_2$, thus
$$
Y_0 = P_1 \ \p_{P_1} + P_2 \ \p_{P_2},
$$
and by the construction of $\lambda_\sigma$, $Y_0 = Y_\sigma$ on $K_c$.
All of these have globally defined flows which dilate the respective
symplectic forms, e.g.~$(\varphi^t_Y)^*\omega = e^t \omega$ 
for all $t \in \RR$.

By the construction of $W^\infty$, there is a smooth function $f : T^3 \to \RR$
such that the closure of $(W^\infty \setminus W,\omega)$ is the trivial 
symplectic cobordism $(\sS_f^\infty,d(e^a\alpha_0))$, and $Y = \p_a$ on this
region.  Now choose $T > 0$ sufficiently large so that
$$
\varphi^T_Y(\p W) \subset [R,\infty) \times T^3,
$$
thus $\varphi^T$ gives a symplectomorphism
$(W,\omega) \to (\varphi^T_Y(W),e^{-T}\omega)$.  Then $\psi \circ \varphi^T_Y$ 
maps $(W,\omega)$ symplectomorphically to the domain in
$(W_\sigma,e^{-T}\omega_\sigma)$ bounded by
$\p \sS_{-\infty}^{f + T} \subset [R,\infty) \times \red{T}^3$,
which is transverse to $Y_\sigma$.  The composition
$$
\psi_T := \varphi^{-T}_{Y_\sigma} \circ \psi \circ \varphi^T_Y :
(W^\infty,\omega) \to (W_\sigma,\omega_\sigma)
$$
now maps $W$ to a compact domain in $W_\sigma$ with boundary transverse 
to~$Y_\sigma$.

Recall next from the proof of Lemma~\ref{lemma:globalPrimitive} that
$\lambda_\sigma = P_1\ dQ_1 + P_2\ dQ_2 + d\widehat{\Phi}$ for some
smooth function $\widehat{\Phi} : W_\sigma \to \RR$ that vanishes in
$K_c$, and we can assume without loss of generality that 
$\Phi(Q_1,Q_2,P_1,P_2)$ depends only on $P_1$ and~$P_2$.  It follows that
$$
Y_\sigma = Y_0 + \widehat{Y}
$$
for some vector field $\widehat{Y}$ that vanishes in $K_c$ and has
components only in the $Q_1$ and $Q_2$--directions.  We can therefore
choose $\tau > 0$ sufficiently large so that 
$\varphi^{-\tau}_{Y_\sigma}$ maps $\psi_T(W)$ into $K_c$ and then
$$
\varphi^{\tau}_{Y_0} \circ \varphi^{-\tau}_{Y_\sigma} 
\red{ : (W_\sigma,\omega_\sigma) \to (W_\sigma,\omega_\sigma) }
$$
is a symplectomorphism that maps
$\psi_T(W)$ to a compact domain with boundary transverse to~$Y_0$.
Under the symplectomorphism $(W_\sigma,\omega_\sigma) \to 
(T^2 \times \RR^2,\omega_0)$ defined by the $(Q,P)$--coordinates, 
this becomes a star shaped domain.  \red{Since all such domains can be
deformed symplectically to the standard filling $(T^2 \times \DD,\omega_0)$,
the uniqueness of strong fillings follows.}
\end{proof}

\subsection{Symplectomorphism groups}
\label{subsec:sympT3}

We now prove Theorem~\ref{thm:sympT3}: \red{observe that by the Whitehead
theorem, it suffices to prove that $\Symp_c(T^*T^2,\omega_0)$ 
is \emph{weakly} contractible,
i.e.~$\pi_n(\Symp_c(T^*T^2,\omega_0)) = 0$ for every $n \ge 0$.}
The main idea of the argument goes back to Gromov \cite{Gromov} in the
closed case, and was also used by Hind \cite{Hind:Lens} in a situation
analogous to ours (fillings of Lens spaces).
The key is to construct a family of foliations by
$J$--holomorphic cylinders for $J$ varying in a ball whose boundary
is determined by a given map $S^n \to \Symp_c(T^*T^2)$.  Here it is
crucial to note that since \red{$\omega_0$ is exact and} 
the closed Reeb orbits in $T^3 = T^2 \times \p \DD$ 
are never contractible in $T^2 \times \DD$, \red{there cannot exist any
closed or $1$--punctured $J$--holomorphic spheres, hence
the moduli spaces we construct have no nodal degenerations.
In this situation,} Theorems~\ref{thm:compactness} and
\ref{thm:compactnessHomotopy} go through without any genericity assumption
for~$J$ (see Remark~\ref{remark:noGenericity}).

As in \S\ref{subsec:SteinModels}, choose an almost complex structure
$J_0$ which matches the standard complex structure near the zero section
and belongs to $\jJ_{\alpha_0}(T^3)$ on the cylindrical end
$[0,\infty) \times T^3$, where it matches the form used in 
Example~\ref{ex:torsion}.  Let $\lambda_0$ denote the canonical
$1$--form on $T^*T^2$, so $d\lambda_0 = \omega_0$.

Suppose now that
$$
S^n \to \Symp_c(T^*T^2,\omega_0) : x \mapsto \psi_x
$$
is a smooth family of symplectomorphisms which all equal the identity
on $[R,\infty) \times T^3$ for some $R \ge 0$, and there is a fixed base
point $x_0 \in S^n$ such that $\psi_{x_0} = \Id$.  
Let $J_x = \psi_x^*J_0$ for each $x \in S^n$: these are all
$\omega_0$--compatible almost complex structures that match $J_0$ on
$[R,\infty)$.  Now using the contractibility of the space of
compatible almost complex structures, the family $\{ J_x \}_{x \in S^n}$
can be filled in to a smooth family $\{ J_x \}_{x \in B^{n+1}}$
that are all compatible with $\omega_0$ and 
equal $J_0$ on $[R,\infty) \times T^3$,
where $B^{n}$ denotes the closed unit ball in~$\RR^n$.

Applying Theorem~\ref{thm:compactnessHomotopy} (with 
Remark~\ref{remark:noGenericity} in mind), 
there are now two unique smooth families of
moduli spaces $\mM_1^x$ and $\mM_2^x$ for $x \in B^{n+1}$, each of which 
consists of embedded $J_x$--holomorphic cylinders foliating $T^*T^2$,
such that each curve in $\mM_1^x$ has one transverse intersection with
each curve in $\mM_2^x$.  We have $J_{x_0} = J_0$, thus the curves
in $\mM_1^{x_0}$ and $\mM_2^{x_0}$ are precisely the cylinders that make
up the splitting
$$
T^*T^2 = T^2 \times \RR^2 = (\RR\times S^1) \times (\RR\times S^1),
$$
as was explained in \S\ref{subsec:SteinModels}.  More generally,
for $x \in \p B^{n+1}$ and $j \in \{1,2\}$, 
the curves in $\mM_j^x$ can be obtained by composing curves in
$\mM_j^{x_0}$ with the symplectomorphism~$\psi_x^{-1}$, and are thus identical
on $[R,\infty) \times T^3$ to the curves in $\mM_j^{x_0}$.
As in the previous
section, we can now use asymptotic evaluation maps to define diffeomorphisms
$$
\RR \times S^1 \to \mM_j^x : (\rho,\eta) \mapsto u_{j,x}^{(\rho,\eta)}.
$$
Arguing further as in Lemma~\ref{lemma:workhorse}, 
for $x \in B^{n+1} \setminus S^n$, change $J_x$ on a region near 
infinity to a smooth family $\hat{J}_x$
tamed by $\omega_0$ and matching $J_0$ on some region 
$[R_2,\infty) \times T^3$, such that for every fixed parameter
$(\rho,\eta)$, the curves $\hat{u}_{j,x}^{(\rho,\eta)}$
in the resulting moduli spaces
$\widehat{\mM}_j^x$ are identical on $[R_2,\infty) \times T^3$ for all 
$x \in B^{n+1}$.  Then the intersection points define a smooth family of
diffeomorphisms 
$$
\psi_x : T^*T^2 \to \widehat{\mM}_1^x \times \widehat{\mM}_2^x = 
(\RR\times S^1) \times (\RR \times S^1) = T^*T^2,
$$
which match the original family $\psi_x \in \Symp_0(T^*T^2,\omega_0)$
for $x \in \p B^{n+1}$ and all equal the identity on 
$[R_2,\infty) \times T^3$.  We have now a smooth family of symplectic
forms $\omega_x := \psi_x^*\omega_0$ which are all standard on
$[R_2,\infty) \times T^3$ and match $\omega_0$ globally for 
$x \in \p B^{n+1}$.

\begin{lemma}
There exists a smooth family of $1$--forms $\{\lambda_x\}_{x \in B^{n+1}}$ 
on $T^*T^2$ such that
\begin{enumerate}
\item
$d\lambda_x = \omega_x$,
\item
$\lambda_x \equiv \lambda_0$ for every $x \in \p B^{n+1}$,
\item
$\lambda_x = \lambda_0$ on $[R_2,\infty) \times T^3$ for every $x \in B^{n+1}$.
\end{enumerate}
\end{lemma}
\begin{proof}
For each $x \in \p B^{n+1}$, $\psi_x$ is a symplectomorphism and thus
$\lambda_0 - \psi_x^*\lambda_0$ is a closed $1$--form with compact support.
All such $1$--forms are exact: indeed, any element of $H_1(T^*T^2)$ can be
represented by a cycle $\gamma$ lying outside the support of
$\lambda_0 - \psi_x^*\lambda_0$, hence
$$
\int_\gamma \left( \lambda_0 - \psi_x^*\lambda_0 \right) = 0
\text{ for all $[\gamma] \in H_1(T^*T^2)$},
$$
implying $[\lambda_0 - \psi_x^*\lambda_0] = 0 \in H^1_{DR}(T^*T^2)$.
Then for $x \in \p B^{n+1}$ there is a unique smooth family of 
compactly supported functions $f_x : T^*T^2 \to \RR$ such that
$$
\lambda_0 = \psi_x^*\lambda_0 + d f_x.
$$
Extending $f_x$ to a smooth family of compactly supported functions for 
$x \in B^{n+1}$, the desired $1$--forms can be defined by
$\lambda_x = \psi_x^*\lambda_0 + d f_x$.
\end{proof}

Now given the $1$--forms $\lambda_x$ from the lemma, define for
$t \in [0,1]$,
$$
\lambda_x^{(t)} := t \lambda_x + (1-t) \lambda_0,
\qquad
\omega_x^{(t)} := d \lambda_x^{(t)}.
$$
The almost complex structure $\hat{J}_x$ is tamed by $\omega_0$, and using
the holomorphic foliations as in the proof of Theorem~\ref{thm:SteinT3},
we see that it is also tamed by $\omega_x = \psi_x^*\omega_0$, and thus by
all $\omega_x^{(t)}$ for $t \in [0,1]$, proving that the latter are symplectic.
Now define a smooth family of time-dependent vector fields $V_x^t$ by
$$
\omega_x^{(t)}(V_x^t,\cdot) = \lambda_0 - \lambda_x.
$$
These vanish identically when $x \in \p B^{n+1}$ and also vanish outside
of a compact set for all~$x$, thus the flows
$\varphi_{V_x}^t$ are well defined and compactly supported 
for all $t$, and trivial if $x \in \p B^{n+1}$.  Moreover,
$(\varphi_{V_x}^t)^*\omega_x^{(t)} = \omega_0$.  We thus obtain a smooth
family of compactly supported symplectomorphisms on $(T^*T^2,\omega_0)$
for $x \in B^{n+1}$ via the composition $\psi_x \circ \varphi_{V_x}^1$,
which matches $\psi_x$ for $x \in \p B^{n+1}$.  This shows that
$\pi_n(\Symp_c(T^*T^2,\omega_0)) = 0$ for all~$n$, and thus completes the
proof of Theorem~\ref{thm:sympT3}.

\appendix

\section{Fredholm and intersection theory}
\label{sec:Fredholm}

\subsection{Transversality}
\label{subsec:transversality}

In this appendix we recall some useful technical facts about finite
energy $J$--holomorphic curves.  Adopting the notation of 
\S\ref{sec:compactness},
$(W^\infty,\omega) = (W,\omega) \cup_{\p W} 
\left( [0,\infty) \times M,d(e^a\lambda) \right)$ is the union of
a compact symplectic manifold $(W,\omega)$ with contact boundary 
$\p W = M$ attached smoothly to the positive cylindrical end
$([0,\infty) \times M, d(e^a\lambda))$, where $\lambda$ is a Morse-Bott
contact form on~$M$, defining the contact structure $\xi = \ker\lambda$.
Let $J$ denote an $\omega$--compatible almost complex structure on $W^\infty$ 
which is in $\jJ_\lambda(M)$
at the positive end.  Then any nonconstant punctured $J$--holomorphic 
curve $u : (\dot{\Sigma},j) \to (W^\infty,J)$ with finite energy is asymptotic
at each puncture $z \in \Gamma$ to some periodic orbit of the
Reeb vector field $X_\lambda$,
for which we can choose a parametrization $x_z : S^1 \to M$ with
$\lambda(\dot{x}_z)$ identically equal to the period $T_z > 0$.
In order to describe the analytical invariants of $u$, it is convenient
to introduce the \emph{asymptotic operators}
$$
\mathbf{A}_z : \Gamma(x_z^*\xi) \to \Gamma(x_z^*\xi) :
v \mapsto - J (\nabla_t v - T_z \nabla_v X_\lambda),
$$
where $\nabla$ is any symmetric connection on~$M$.  Morally, this is the
Hessian of the contact action functional on $C^\infty(S^1,M)$, whose 
critical points are periodic orbits; in particular one can show that 
$\mathbf{A}_z$ has trivial kernel if and only if
the orbit $x_z$ is nondegenerate.  Choosing a unitary trivialization
$\Phi$ for $x_z^*\xi$, $\mathbf{A}_z$ becomes identified with the
operator
$$
C^\infty(S^1,\RR^2) \to C^\infty(S^1,\RR^2) : v \mapsto
-J_0 \dot{v} - S v
$$
where $S(t)$ for $t \in S^1$ is a smooth loop of symmetric $2$--by--$2$
matrices.  Then there is a linear Hamiltonian flow $\Psi(t) \in \Spp(1)$
defined by solutions to the equation $-J_0 \dot{v} - Sv = 0$, and
$1$ is in the spectrum of $\Psi(1)$ if and only if $\ker\mathbf{A}_z$ is
nontrivial.  When this is not the case, we define the \emph{Conley-Zehnder
index} $\muCZ^\Phi(\mathbf{A}_z)$ in the standard way in terms of this
path of symplectic matrices for $t \in [0,1]$ \red{(cf.~the discussion of the
``$\mu$--index'' in \cite{HWZ:props2}*{\S 3})}.  
Note that the index depends on $\Phi$ 
up to an even integer, so its even/odd
parity in particular is independent of~$\Phi$.
In the Morse-Bott context, $\mathbf{A}_z$ may have nontrivial kernel,
but one can generally pick a real number $\epsilon \ne 0$ and define
$\muCZ^\Phi(\mathbf{A}_z + \epsilon)$,
which depends only on the sign of $\epsilon$ if the latter is
sufficiently close to zero.

The \emph{Fredholm index} of $u$ can now be written as
\begin{equation}
\label{eqn:index}
\ind(u) = -\chi(\dot{\Sigma}) + 2 c_1^\Phi(u^*TW^\infty) + \sum_{z \in \Gamma}
\muCZ^\Phi(\mathbf{A}_z - \epsilon),
\end{equation}
where $\epsilon > 0$ is an arbitrary small number, and
$c_1^\Phi(u^*TW^\infty)$ is the \emph{relative first Chern number} of the
complex vector bundle $(u^*TW^\infty,J)$ with respect to the trivialization at
the ends defined by combining $\Phi$ on $\xi$ with the obvious trivialization
of $\RR \oplus \RR X_\lambda$.  It is straightforward to show from
properties of the Conley-Zehnder index and relative Chern number that this
sum doesn't depend on either $\epsilon$ or $\Phi$.  It defines the
\emph{virtual dimension} of the moduli space of $J$--holomorphic curves close
to~$u$.  We say that $u$ is \emph{unobstructed} whenever 
the linearized Cauchy-Riemann operator at $u$ is surjective: then the
moduli space close to~$u$ is a smooth orbifold (or manifold if $u$ is 
somewhere injective) of dimension $\ind(u)$.
In the case where all orbits are nondegenerate, this follows 
from the Fredholm theory developed in \cite{Dragnev}; see \cite{Wendl:thesis}
or \cite{Wendl:BP1} for the Morse-Bott case.

The punctures $\Gamma \subset \Sigma$ can be divided into \emph{even}
punctures $\Gamma_0$ and \emph{odd} punctures $\Gamma_1$ according to the
parity of $\muCZ^\Phi(\mathbf{A}_z - \epsilon)$, which is independent of
$\Phi$ and $\epsilon > 0$ as noted above.\footnote{Note that we're 
assuming all punctures are
positive here; if there were negative Morse-Bott punctures, both this
definition of parity and the Fredholm index formula would need
$\mathbf{A}_z + \epsilon$ instead of $\mathbf{A}_z - \epsilon$.}
Now one can easily use the index formula to show that $\ind(u)$ and $\Gamma_0$
are either both even or both odd, so if $\Sigma$ has genus $g$,
there is an integer $c_N(u) \in \ZZ$ defined by the formula
\begin{equation}
\label{eqn:cNindex}
2 c_N(u) = \ind(u) - 2 + 2g + \#\Gamma_0.
\end{equation}
We call this the \emph{normal Chern number} of $u$, for reasons that are
easy to see in the case where $W$ is a closed manifold: then the
combination of \eqref{eqn:index} and \eqref{eqn:cNindex} yields the
alternative definition $c_N(u) = c_1(u^*TW) - \chi(\Sigma)$, which is
precisely the first Chern number of the normal bundle whenever $u$ is
immersed.  As shown in \cite{Wendl:automatic}, this is also the appropriate
interpretation of $c_N(u)$ in the punctured case.  The following 
transversality criterion is a special case of a result proved in
\cite{Wendl:automatic}:

\begin{prop}
\label{prop:automatic}
If $u : \dot{\Sigma} \to W^\infty$ is an immersed finite energy $J$--holomorphic
curve with $\ind(u) > c_N(u)$, then $u$ is unobstructed.
\end{prop}

A stronger statement holds in the case where $u$ is embedded with all
asymptotic orbits distinct and simply covered, $\ind(u) = 2$ and
$c_N(u) = 0$.  Then a result in \cites{Wendl:thesis,Wendl:BP1} shows that
the smooth $2$--dimensional moduli space of curves near $u$ foliates a
neighborhood of $u(\dot{\Sigma})$ in~$W^\infty$.  The reason is that
tangent vectors to the moduli space can be identified with sections of
the normal bundle $N_u \to \dot{\Sigma}$ that satisfy a linear Cauchy-Riemann
type equation, and the condition $c_N(u) = 0$ constrains these sections to
be \red{nowhere zero}.  It follows that if we
add one marked point and consider the resulting evaluation map from the
moduli space into $W^\infty$, this map is a local diffeomorphism.

\subsection{Asymptotic evaluation maps}
\label{subsec:evaluation}

For the arguments in \S\ref{sec:T3}, it is convenient to have an asymptotic
version of the above statement about the evaluation map.  Consider a
connected moduli space $\mM$ of finite energy $J$--holomorphic curves
$u : \dot{\Sigma} \to W^\infty$ that each have an odd puncture
asymptotic to an orbit $x : S^1 \to M$ belonging
to a $1$--parameter family $\pP$ of simply covered Morse-Bott orbits
of period $T > 0$.  To simplify the notation, we'll assume this is the only
puncture, though the discussion can be generalized to multiple punctures in an
obvious way.
Let $\mathbf{A}_x$ denote the asymptotic operator for any $x \in \pP$;
since it is a $1$--parameter family, $\dim\ker\mathbf{A}_x = 1$.  We will
use certain facts about the \red{spectrum $\sigma(\mathbf{A}_x)$} 
of $\mathbf{A}_x$ that \red{are proved}
in \cite{HWZ:props2}: in particular, \red{for any nontrivial
eigenfunction $e_\lambda \in \Gamma(x^*\xi)$ with eigenvalue $\lambda$,
the winding number $\wind^\Phi(\lambda) := \wind^\Phi(e_\lambda) \in \ZZ$ 
depends only on $\lambda$, so that the resulting function
$$
\sigma(\mathbf{A}_x) \to \ZZ : \lambda \mapsto \wind^\Phi(\lambda)
$$
is monotone and attains every integer value exactly twice (counting
multiplicity of eigenvalues).  If $0 \not\in\sigma(\mathbf{A}_x)$, then
one can also deduce the parity of
$\muCZ^\Phi(\mathbf{A}_x)$ from these winding numbers: it is even if and only
if $\sigma(\mathbf{A}_x)$ contains a positive and negative eigenvalue for
which the winding numbers match.  It follows that if
$\muCZ^\Phi(\mathbf{A}_x - \epsilon)$ is odd and
$\lambda_x < 0$ is the largest negative eigenvalue of $\mathbf{A}_x$,}
then the corresponding eigenspace $E_x \subset \Gamma(x^*\xi)$ is 
$1$--dimensional and its eigenfunctions have zero winding relative to 
any nonzero element of $\ker\mathbf{A}_x$.  The union of these eigenspaces for
all $x \in \pP$ defines a real line bundle
$$
E \to \pP.
$$
The eigenfunctions of $\mathbf{A}_x$
appear naturally in the asymptotic formula proved in \cite{HWZ:props4}
\red{(see also \cite{Siefring:asymptotics} for a fuller discussion)}
for a map $u \in \mM$ asymptotic to $x_u \in \pP$.  
Choose coordinates $(s,t) \in [0,\infty) \times S^1$ for a 
neighborhood of the puncture in $\dot{\Sigma}$, and assume without loss 
of generality that $u$ maps this neighborhood into $[0,\infty) \times M$.  
Then using any $\RR$--invariant connection to define the exponential
map, one can choose the coordinates $(s,t)$ so that for sufficiently large
$s$, $u$ satisfies
\begin{equation}
\label{eqn:asympFormula}
u(s,t) = \exp_{(T s,x_u(t))} \left[ e^{\lambda_x s} \left(
f_u(t) + r_u(s,t) \right) \right],
\end{equation}
where $f_u \in E_x$ and $r_u(s,t) \in \xi_{x_u(t)}$ is smooth and converges
to~$0$ uniformly in $t$ as $s \to \infty$.  This formula defines an
``asymptotic evaluation map''
$$
\ev : \mM \to E : u \mapsto (x_u,f_u).
$$

\begin{lemma}
\label{lemma:evalInfty}
In the situation described above, if $u \in \mM$ is immersed with $\ind(u) = 2$
and $c_N(u) = 0$, then $\ev : \mM \to E$ is a local diffeomorphism 
near~$u$.
\end{lemma}
\begin{proof}
We will use the analytical setup in \cite{Wendl:automatic} to show that
under these conditions, $d\ev(u) : T_u\mM \to
T_{(x_u,f_u)}E$ is nonsingular.  If $N_u \to \dot{\Sigma}$ denotes the normal
bundle of $u$, $p > 2$ and $\epsilon > 0$ is small, we have
$T_u \mM = \ker\mathbf{D}_u^N$, where 
$$
\mathbf{D}_u^N : W^{1,p,-\epsilon}(N_u) \to
L^{p,-\epsilon}(\overline{\Hom}_\CC(T\dot{\Sigma},N_u))
$$ 
is the \emph{normal Cauchy-Riemann operator}, defined on exponentially
weighted Sobolev spaces
$$
W^{k,p,-\epsilon} := \{ v \in W^{k,p}_{\text{loc}}\ |\ 
e^{-\epsilon s} v(s,t) \in W^{k,p}([0,\infty)\times S^1) \}
$$
for $k = \{0,1\}$.
Note that by Prop.~\ref{prop:automatic},
$u$ is unobstructed and thus $\dim \ker\mathbf{D}_u^N = 2$.
\red{By an asymptotic version of local elliptic regularity
(see \cites{HWZ:props1,Siefring:asymptotics}), any section
$v \in \ker\mathbf{D}_u^N$ satisfies a linearized version of
\eqref{eqn:asympFormula} in the form
\begin{equation}
\label{eqn:linearAsymp}
v(s,t) = e^{\lambda s} (f_v(t) + r(s,t)),
\end{equation}
where $f_v \in \Gamma(x_u^*\xi)$ is an
eigenfunction of $\mathbf{A}_{x_u}$ with eigenvalue $\lambda < \epsilon$,
and $r(s,t) \to 0$ as $s \to \infty$.  In the present situation, the
largest eigenvalue less than $\epsilon$ is~$0$, thus if $v$ is
nontrivial then $\wind^\Phi(f_v) \le \wind^\Phi(0)$.  The zeroes of $v$
are then isolated and positive, and can be counted by the normal Chern number: 
we have
\begin{equation}
\label{eqn:cNZ}
Z(v) + Z_\infty(v) = c_N(u),
\end{equation}
where $Z(v)$ is the algebraic count of zeros of $v$, and $Z_\infty(v)$ is
a corresponding asymptotic contribution defined as
$\wind^\Phi(0) - \wind^\Phi(f_v)$, and is thus also nonnegative.
So the condition $c_N(u) = 0$ implies that $f_v$ has winding number zero
relative to any nontrivial section in $\ker\mathbf{A}_{x_u}$.
}

We can consider also the restriction of $\mathbf{D}_u^N$ to a smaller
weighted domain,
$$
\mathbf{D}' : W^{1,p,\epsilon}(N_u) \to 
L^{p,\epsilon}(\overline{\Hom}_\CC(T\dot{\Sigma},N_u)),
$$
which amounts to linearizing the $J$--holomorphic curve problem with an added
constraint fixing the asymptotic orbit at the puncture.  This operator has
index~$1$ and is also surjective, 
by the results in \cite{Wendl:automatic}.  It follows
that there is a unique one-dimensional subspace $V_u \subset T_u\mM$ 
consisting of sections $v \in \ker\mathbf{D}_u^N$
for which the eigenvalue $\lambda$ in \eqref{eqn:linearAsymp} is negative.
For all $v \in \ker\mathbf{D}_u^N \setminus V_u$, this eigenvalue is
zero, and we thus have
$v(s,\cdot) \to f_v \in \ker\mathbf{A}_{x_u}$ as $s \to \infty$, implying
that the derivative of the map $\mM \to \pP : u \mapsto x_u$ in this direction
is nonzero.

Now fix an orbit $x \in \pP$ and let $\mM_x = \{ u \in \mM\ |\ x_u = x \}$.
By the remarks above, this is a $1$--dimensional submanifold with
$T_u \mM_x = V_u$.  The restriction of $\ev$ to $\mM_x$ defines
a map $\mM_x \to E_x$, and we claim finally that for any
nontrivial $v \in V_u$, the directional derivative of this map is
nonzero.  This follows from \eqref{eqn:linearAsymp} and the fact that
$Z_\infty(v) = 0$, as the nontrivial eigenfunction in \eqref{eqn:linearAsymp}
must have the same winding as a section in $\ker\mathbf{A}_{x_u}$, and
therefore belongs to $E_{x_u}$.
\end{proof}

\subsection{Intersection numbers}
\label{subsec:intersections}

We discuss next the punctured generalization 
of the adjunction formula.  \red{These results are the topological
consequences of the relative asymptotic analysis carried out by Siefring in
\cite{Siefring:asymptotics}; complete details are explained in
\cite{Siefring:intersection} for
curves with nondegenerate orbits and \cite{SiefringWendl} for the 
Morse-Bott case, and a summary with precise definitions
may also be found in the last section of 
\cite{Wendl:automatic}}.  We shall only need a few details, which we now
state without proof.  For any two finite energy curves $u_1, u_2$, there
exists an \emph{intersection number}
$$
i(u_1 ; u_2) \in \ZZ
$$
which algebraically counts actual intersections plus a certain
``asymptotic contribution,'' which vanishes generically.  \red{The asymptotic
contribution vanishes in particular whenever $u_1$ and $u_2$ have no 
asymptotic orbits in common, and it is otherwise analogous to the term
$Z_\infty(v)$ in \eqref{eqn:cNZ}: it is a nonnegative measure of the winding 
numbers of certain asymptotic eigenfunctions that describe the relative 
approach of two distinct curves to the same orbit, and it vanishes if and 
only if these winding numbers attain the extremal values determined by
the spectrum.  Thus if
$u_1$ and $u_2$ do not cover the same somewhere injective curve,} both the
actual intersection count and the asymptotic contribution are nonnegative, and
moreover, their sum is invariant under deformations of both curves through the
moduli space.  The condition $i(u_1 ; u_2) = 0$ then suffices to ensure that
$u_1$ and $u_2$ never have isolated intersections.  
For any somewhere injective curve $u$,
there is also a \emph{singularity number} $\sing(u) \in \ZZ$, which counts double
points, critical points and ``asymptotic singularities,'' each 
contributing nonnegatively.  This sum is also invariant under deformations,
and the condition $\sing(u) = 0$ suffices to ensure that a somewhere injective
curve is embedded.  The standard adjunction formula for closed holomorphic
curves now generalizes to
\begin{equation}
\label{eqn:adjunction}
i(u ; u) = 2\sing(u) + c_N(u) + \sum_{z\in\Gamma} \cov_\infty(z),
\end{equation}
where the terms $\cov_\infty(z)$ are nonnegative integers that
\red{vanish whenever certain asymptotic eigenfunctions are simply covered,
so they}
depend only on the asymptotic orbit and sign of the respective puncture
$z \in \Gamma$.

Finally, we observe one relevant situation where the left hand side of
\eqref{eqn:adjunction} is guaranteed to be zero.  The proof below is only a
sketch; we refer to \cite{Siefring:intersection} for details.

\begin{lemma}
\label{lemma:i0}
Suppose that $u : \dot{\Sigma} \to W^\infty$ and
$u' : \dot{\Sigma}' \to W^\infty$ are finite energy 
$J$--holomorphic curves that are both contained in $[0,\infty) \times M$ 
and have embedded projections to $M$ that are either identical
or disjoint.  If also $c_N(u) = 0$, then $i(u;u') = 0$.
\end{lemma}
\begin{proof}
The almost complex structure is $\RR$--invariant in the region containing
$u$ and $u'$, thus after translating $u'$ upwards, we can assume without
loss of generality that $u$ and $u'$ have no intersections.  This
$\RR$-translation changes the asymptotic eigenfunctions at the ends of $u'$
by multiplication with a positive number, thus we can also assume these
eigenfunctions are not identical at any common asymptotic orbit 
of $u$ and~$u'$.  Now the vanishing of $c_N(u)$ implies due to
$\RR$--invariance that $u$ has no \emph{asymptotic defect}
(cf.~\cite{Wendl:compactnessRinvt}): this means its asymptotic eigenfunctions
all attain the largest allowed winding number.  The asymptotic
analysis of \cite{Siefring:asymptotics} then implies that the same is
true for the eigenfunctions controlling the relative behavior of $u$ and $u'$
at infinity, so the asymptotic contribution to $i(u;u')$ is zero.
\end{proof}

\end{document}